\documentclass[11pt]{amsart}

\usepackage{a4wide}
\usepackage{graphicx}
\usepackage{amssymb}
\usepackage{epstopdf}
\usepackage{enumerate}
\usepackage{titletoc}
\usepackage[colorlinks=true, urlcolor=blue, linkcolor=blue, citecolor=black]{hyperref}
\usepackage[nameinlink]{cleveref}
\usepackage{esint}

\theoremstyle{plain}
\newtheorem{theorem}{Theorem}[section]
\newtheorem{lemma}[theorem]{Lemma}

\newtheorem{proposition}[theorem]{Proposition}

\theoremstyle{definition}
\newtheorem{definition}[theorem]{Definition}

\newtheorem{example}[theorem]{Example}

\numberwithin{equation}{section}

\newcommand{\Rn}{{\mathbb{R}^n}}
\newcommand{\ve}{\vert}
\newcommand{\lv}{\left\vert}
\newcommand{\rv}{\right\vert}
\newcommand{\Ve}{\Vert}

\newcommand{\lb}{\left\lbrace}
\newcommand{\rb}{\right\rbrace}
\newcommand{\M}{\mathcal{M}}
\newcommand{\I}{\mathcal{I}}

\title[Generalized Evans--Krylov and Schauder type estimates]{Generalized Evans--Krylov and Schauder type estimates for nonlocal fully nonlinear equations with rough kernels of variable orders}

\author{Minhyun Kim}
\address{Department of Mathematical Sciences, Seoul National University, Seoul 08826, Republic of Korea}
\email{201421187@snu.ac.kr}

\author{Ki-Ahm Lee}
\address{Department of Mathematical Sciences, Seoul National University, Seoul 08826, Republic of Korea \& Korea Institute for Advanced Study, Seoul 02455, Republic of Korea}
\email{kiahm@snu.ac.kr}


\begin{document}

\begin{abstract}
We establish the generalized Evans–Krylov and Schauder type estimates for nonlocal fully nonlinear elliptic equations with rough kernels of variable orders. In contrast to the fractional Laplacian type operators having a fixed order of differentiability $\sigma \in (0,2)$, the operators under consideration have variable orders of differentiability. Since the order is not characterized by a single number, we consider a function $\varphi$ describing the variable orders of differentiability, which is allowed to oscillate between two functions $r^{\sigma_1}$ and $r^{\sigma_2}$ for some $0 < \sigma_1 \leq \sigma_2 < 2$. By introducing the generalized H\"older spaces, we provide $C^{\varphi\psi}$ estimates that generalizes the standard Evans–Krylov and Schauder type $C^{\sigma+\alpha}$ estimates.
\end{abstract}

\maketitle

\section{Introduction}

This paper is concerned with the Evans–Krylov type and the Schauder type generalized H\"older estimates for nonlocal fully nonlinear equations with rough kernels of variable orders. We first provide the Evans–Krylov type interior estimates for concave translation invariant elliptic equations with respect to the class $\mathcal{L}_0(\varphi)$ of linear operators with rough kernels of variable orders, where a function $\varphi$ is allowed to oscillate between two functions $r^{\sigma_1}$ and $r^{\sigma_2}$ for some $0 < \sigma_1 \leq \sigma_2 < 2$. We next establish the Schauder type estimates for equations having $x$ dependence in a generalized H\"older fashion. All the regularity estimates are obtained in much finer scale of H\"older space $C^{\varphi \psi}$, and recover the classical Evans–Krylov theorem and Schauder theorem for second order fully nonlinear equations as limits. Moreover, we do not restrict ourselves to the Bellman type operators, but consider nonlinear operators in full generality.

Let us first explain a close connection between the theory of stochastic processes and the theory of partial differential (and integro-differential) equations that motivates this work. The infinitesimal generator of a pure jump L\'evy process is given by
\begin{equation*}
Lu(x) = \int_{\Rn \setminus \lbrace 0 \rbrace} \left( u(x+y) - u(x) - \chi_{\lbrace \ve y \ve \leq 1 \rbrace} y \cdot Du(x) \right) \mu(dy),
\end{equation*}
where $\mu$ is the L\'evy measure. This formula has motivated both studies on pure jump stochastic processes and on integro-differential equations. The simplest example is the rotationally symmetric $\sigma$-stable process whose infinitesimal generator is the fractional Laplacian $-(-\Delta)^{\sigma/2}$. However, there is a large family of L\'evy processes, known as subordinate Brownian motions, that exhibits numerous interesting examples of nonlocal operators, such as sums of symmetric stable processes, relativistic stable processes, and geometric stable processes. Their infinitesimal generators have kernels of variable orders and this is why we study the integro-differential equations of variable orders. See \cite{CS98, BL02, BK05a, BK05b, BGR14, CKK11, Mim12} for the exposition to regularity results of linear equations using probabilistic methods.

On the other hand, the PDE approaches enable us to extend the regularity theory to nonlinear integro-differential equations. Especially, the class of integro-differential operators mentioned above has been studied recently \cite{Bis18, BL19, BJ20}. In this work let us focus on the Evans–Krylov theory and the Schauder theory, which have been studied only for equations of fixed order. The Evans–Krylov type $C^{\sigma+\alpha}$ interior estimate for nonlocal equations was first established by Caffarelli and Silvestre \cite{CS11a}. This result states that if $u$ is a bounded solution of $\inf_{L \in \mathcal{L}_2} Lu = 0$ in $B_1$, then $\Ve u \Ve_{C^{\sigma+\alpha}(\overline{B_{1/2}})} \leq C \Ve u \Ve_{L^\infty(\Rn)}$. Here $\mathcal{L}_2 = \mathcal{L}_2(\sigma)$ is the class of translation invariant linear operators of the form
\begin{equation} \label{eq:linear_operator}
Lu(x) = \int_\Rn (u(x+y) + u(x-y) - 2u(x)) K(y) dy,
\end{equation}
where kernels $K$ satisfy the ellipticity condition
\begin{equation} \label{eq:ellipticity_FL}
\lambda \frac{(2-\sigma)}{\ve y \ve^{n+\sigma}} \leq K(y) \leq \Lambda \frac{(2-\sigma)}{\ve y \ve^{n+\sigma}}
\end{equation}
with ellipticity constants $0 < \lambda \leq \Lambda$ and the constant order $\sigma \in (0,2)$, and the scaling invariant bounds for all its second order partial derivatives:
\begin{equation} \label{eq:class_L2}
[K]_{C^2(\Rn \setminus B_r)} \leq \Lambda (2-\sigma) r^{-n-\sigma-2} \quad\text{for all} ~ r > 0.
\end{equation}
Note that the symmetry of the kernels is encoded in the expression \eqref{eq:linear_operator}. See \cite{KL17} for the Evans–Krylov estimate for parabolic equations. In the paper \cite{Ser15a}, Serra improved this result in \cite{CS11a} to the equations with rough kernels. More precisely, he proved that if $u \in C^{\sigma+\alpha}(B_1) \cap C^\alpha (\Rn)$ solves $\inf_{L \in \mathcal{L}_0} Lu = 0$ in $B_1$, then $\Ve u \Ve_{C^{\sigma+\alpha}(\overline{B_{1/2}})} \leq C \Ve u \Ve_{C^\alpha(\Rn)}$, where $\mathcal{L}_0 = \mathcal{L}_0(\sigma)$ is the class of linear operators of the form \eqref{eq:linear_operator} with kernels satisfying \eqref{eq:ellipticity_FL}, but not necessarily \eqref{eq:class_L2}.

Schauder estimates have been established for linear integro-differential operators \cite{Bas09, DK13, ROS14, BK15} and Bellman type integro-differential operators \cite{Ser15a, JX16} in different contexts. In \cite{Ser15a}, it was shown that the proof for the Evans–Krylov estimates also works for the equations having $x$ dependence in $C^\alpha$ fashion. He proved that if $u \in C^{\sigma+\alpha}(B_1) \cap C^\alpha (\Rn)$ is a solution to a non-translation invariant equation
\begin{equation} \label{eq:Bellman_operator}
\inf_{a \in \mathcal{A}} \left( \int_\Rn (u(x+y) + u(x-y) - 2u(x)) K_a(x, y) dy + c_a (x) \right) = 0 \quad\text{in}~ B_1,
\end{equation}
where $\mathcal{A}$ is some index set, $K_a$ are kernels satisfying \eqref{eq:ellipticity_FL} and 
\begin{equation} \label{eq:K_Holder}
\int_{B_{2r} \setminus B_r} \ve K_a(x, y) - K_a(x', y) \ve dy \leq A_0 \ve x-x' \ve^\alpha \frac{2-\sigma}{r^\sigma} \quad\text{for all}~ x, x' \in \Rn, r > 0,
\end{equation}
and $c_a$ are functions with $\Ve c_a \Ve_{C^\alpha(\overline{B_1})} \leq C_0$, then $\Ve u \Ve_{C^{\sigma+\alpha}(\overline{B_{1/2}})} \leq C(C_0 + \Ve u \Ve_{C^\alpha(\Rn)})$. Moreover, it was proved that if the kernels $K_a$ additionally satisfy
\begin{equation} \label{eq:class_Lalpha}
[K_a(x, \cdot)]_{C^\alpha(\Rn \setminus B_r)} \leq \Lambda(2-\sigma) r^{-n-\sigma-\alpha} \quad\text{for all}~ r > 0,
\end{equation}
then the uniform estimate $\Ve u \Ve_{C^{\sigma+\alpha}(\overline{B_{1/2}})} \leq C(C_0 + \Ve u \Ve_{L^\infty(\Rn)})$ holds for merely bounded solutions. The subclass of $\mathcal{L}_0$ consisting of linear operators whose kernels satisfy \eqref{eq:class_Lalpha} is denoted by $\mathcal{L}_\alpha$.

On the other hand, the Schauder type $C^{\sigma+\alpha}$ estimate was also established independently by Jin and Xiong \cite{JX16}. They proved that bounded solutions to Bellman type equations, with smooth kernels in $\mathcal{L}_2$ and $C^\alpha$ dependence on $x$, have uniform estimates. The assumption $\mathcal{L}_2$ is stronger than $\mathcal{L}_\alpha$, but their proof is very different from the proof of Serra.

All the aforementioned results concerning the Evans–Krylov and the Schauder estimates are dealt with integro-differential equations for a fixed order of differentiability $\sigma \in (0,2)$, except for \cite{BK15}, where the Schauder theory for linear integro-differential operators of variable orders is obtained through the potential theory. The aim of this paper is to establish the generalized Evans–Krylov and Schauder $C^{\varphi\psi}$ interior estimates for general nonlocal fully nonlinear equations with rough kernels of variable orders. Our proofs are significantly different from the proof in \cite{BK15} because the equations we consider are nonlinear. Moreover, as we mentioned before, we do not restrict ourselves to the Bellman type operators \eqref{eq:Bellman_operator} and provide the regularity results in full generality without assuming an explicit form of operators.

\subsection{Nonlocal operators of variable orders}

We are mainly concerned with nonlocal fully nonlinear equations with rough kernels of variable orders. The definitions of these notions are made precise in this section.

We say that a function $\varphi : (0, +\infty) \to (0, +\infty)$ satisfies the {\it weak scaling property} with constants $0 < \sigma_1 \leq \sigma_2 < 2$ and $a \geq 1$, if
\begin{equation} \label{a:WS}
a^{-1} \left( \frac{R}{r} \right)^{\sigma_1} \leq \frac{\varphi(R)}{\varphi(r)} \leq a \left( \frac{R}{r} \right)^{\sigma_2} \quad\text{for all}~ 0 < r \leq R.
\end{equation}
Notice that $\varphi$ generalizes the polynomial $r^\sigma$, $\sigma \in (0,2)$. Since $\varphi$ is allowed to oscillate between two functions $r^{\sigma_1}$ and $r^{\sigma_2}$, it is referred to as {\it variable orders}. Throughout the paper, except for \Cref{section:WS}, we always assume that $\varphi$ is a function satisfying $\varphi(1) = 1$, the weak scaling property \eqref{a:WS}, and
\begin{equation} \label{a:varphi}
r \varphi'(r) \leq C \varphi(r),
\end{equation}
and that the function $\phi$ defined by $\phi(r) = \varphi(r^{-1/2})^{-1}$ is a Bernstein function, i.e., $\phi$ satisfies $\phi(r) \geq 0$ and $(-1)^k \phi^{(k)}(r) \leq 0$ for every $k \in \mathbb{N}$. The last assumption is for a later use of results in \cite{BK15}. Here are some examples of Bernstein functions satisfying the above assumptions.

\begin{example}
\begin{enumerate} [(i)]
\item
$\phi(r) = r^{\sigma/2}$, $\sigma \in (0,2)$. In this case, $\sigma_1 = \sigma_2 = \sigma$.
\item
$\phi(r) = r^{\sigma_1/2} + r^{\sigma_2/2}$, $0 < \sigma_1 \leq \sigma_2 < 2$.
\item
$\phi(r) = (r+m^{2/\sigma})^{\sigma/2} - m$, $\sigma \in (0,2)$, $m \geq 0$. In this case, $\sigma_1 = \sigma$ and $\sigma_2 = 1$.
\item
$\phi(r) = r^{\sigma_1/2} (\log (1+r))^{\sigma_2/2-\sigma_1/2}$ for $\sigma_1, \sigma_2, \sigma_2-\sigma_1 \in (0,2)$.
\item
$\phi(r) = r^{\sigma_2/2} (\log(1+r))^{\sigma_1/2-\sigma_2/2}$, $0 < \sigma_1 \leq \sigma_2 < 2$.
\end{enumerate}
\end{example}

We consider a class $\mathcal{L}$ of linear integro-differential operators of the form \eqref{eq:linear_operator}. Using the extremal operator defined by
\begin{equation*}
\M^+_{\mathcal{L}} u(x) = \sup_{L \in \mathcal{L}} Lu(x) \quad\text{and}\quad \M^-_{\mathcal{L}} u(x) = \inf_{L \in \mathcal{L}} Lu(x),
\end{equation*}
we are going to impose the ellipticity to a nonlocal fully nonlinear operator in the standard way.

\begin{definition} \label{def:ellipticity}
Let $\mathcal{L}$ be a class of linear integro-differential operators. An operator $\I$ is said to be {\it elliptic with respect to the class $\mathcal{L}$} if
\begin{equation*}
\M^-_{\mathcal{L}} (u-v) (x) \leq \I(u, x) - \I(v, x) \leq \M^+_{\mathcal{L}} (u-v)(x).
\end{equation*}
\end{definition}

The operator $\I$ we consider in this paper may be translation invariant or non-translation invariant. In the former case, we write $\I u(x)$ instead of $\I (u, x)$.

We consider, in particular, the class $\mathcal{L}_0(\varphi)$ of linear integro-differential operators of the form \eqref{eq:linear_operator}, where kernels $K$ satisfy
\begin{equation} \label{a:ellipticity}
\lambda \frac{c_\varphi}{\ve y \ve^n \varphi(\ve y \ve)} \leq K(y) \leq \Lambda \frac{c_\varphi}{\ve y \ve^n \varphi(\ve y \ve)}
\end{equation}
with ellipticity constants $0 < \lambda \leq \Lambda$. The constant $c_\varphi$ in \eqref{a:ellipticity} is given by
\begin{equation} \label{eq:c_varphi}
c_\varphi = \left( \int_0^1 \frac{r}{\varphi(r)} dr \right)^{-1},
\end{equation}
and it plays an important role in the uniform estimates as $2-\sigma$ does for the case of the fractional Laplacian. Notice that $c_\varphi = 2-\sigma$ when $\varphi(r) = r^\sigma$. In \cite{KL20}, the authors used a constant
\begin{equation} \label{eq:full_constant}
\left( \int_\Rn \frac{1-\cos y_1}{\ve y \ve^n \varphi(\ve y \ve)} dy \right)^{-1},
\end{equation}
which corresponds to the full constant $C(n, \sigma)$ of the fractional Laplacian, instead of the constant \eqref{eq:c_varphi}, for the Krylov–Safonov theory. However, since we are not interested in the limit behavior $\sigma \to 0$, it is enough to consider \eqref{eq:c_varphi} instead of \eqref{eq:full_constant}, and the use of the constant \eqref{eq:c_varphi} will shorten the proofs. With respect to the class $\mathcal{L}_0(\varphi)$, the extremal operators have the explicit form
\begin{align*}
\M^+_{\mathcal{L}_0(\varphi)} u(x) 
&= \int_\Rn \left( \Lambda \delta_+(u, x, y) - \lambda \delta_-(u, x, y) \right) \frac{c_\varphi}{\ve y \ve^n \varphi(\ve y \ve)} dy \quad\text{and} \\
\M^-_{\mathcal{L}_0(\varphi)} u(x) 
&= \int_\Rn \left( \Lambda \delta_-(u, x, y) - \lambda \delta_+(u, x, y) \right) \frac{c_\varphi}{\ve y \ve^n \varphi(\ve y \ve)} dy,
\end{align*}
where $\delta(u, x, y) = u(x+y) + u(x-y) - 2u(x)$ is the second order incremental quotients.

The operator we are going to consider in this paper is two-fold. For the Evans–Krylov type estimates, we will assume that $\I$ is a concave translation invariant elliptic operator with respect to the class $\mathcal{L}_0(\varphi)$. A typical example is the Bellman type operator, but a novelty of this work with respect to \cite{Ser15a} and \cite{JX16} is that the proof does not rely on an explicit form of the operator.

For the Schauder type estimates, we will consider more general operators $\I (u,x)$ which are not necessarily translation invariant. The standard assumptions we need to impose on $\I$ are that $\I$ has an $x$ dependence in H\"older fashion for the ``freezing coefficients" step and that the model equation obtained by freezing coefficients has an appropriate regularity estimate. Recall that, in \cite{Ser15a} and \cite{JX16}, the $C^\alpha$ dependence \eqref{eq:K_Holder} in $x$ variable is imposed to kernels of the operator. However, since we do not assume the explicit form of the operator, the H\"older dependence in $x$ variable must be imposed directly to the operator (see \eqref{a:Holder}). Moreover, the $x$ dependence of equation will be given in a generalized H\"older fashion.

\subsection{Main results}

In order to state the main results, we briefly discuss the concept of order of differentiability. Throughout the paper, we will always assume that $\psi : (0, +\infty) \to (0, +\infty)$ is a function satisfying
\begin{equation} \label{a:psi}
\psi(1) = 1 \quad\text{and}\quad \lim_{r \to +0} \psi(r) = 0.
\end{equation}
We say that $\psi$ is {\it almost increasing} if there is a constant $c \in (0,1]$ such that $c \psi(r) \leq \psi(R)$ for all $r \leq R$. Similarly, we say that $\psi$ is {\it almost decreasing} if there is $C \in [1, \infty)$ such that $\psi(R) \leq C \psi(r)$ for all $r \leq R$. 
We adopt, from \cite{BK15}, the definition of indices $M_\psi$ and $m_\psi$, which is given by
\begin{align} \label{eq:index_psi}
\begin{split}
M_\psi &= \inf \lb \alpha \in \mathbb{R} : r \mapsto r^{-\alpha} \psi(r) ~\text{is almost decreasing} \rb, \\
m_\psi &= \sup \lb \alpha \in \mathbb{R} : r \mapsto r^{-\alpha} \psi(r) ~\text{is almost increasing} \rb,
\end{split}
\end{align}
and denote by $I_\psi$ the closed interval $[m_\psi, M_\psi]$. The interval $I_\psi$ describes the range of orders of differentiability induced by $\psi$. Note that if $\psi(r) = r^\alpha$ or $\psi(r) = r^\alpha \ve \log(2/r) \ve$, then $M_\psi = m_\psi = \alpha$, and if $\psi(r) = r^\alpha+r^\beta$, then $M_\psi = \max \lbrace \alpha, \beta \rbrace$ and $m_\psi = \min \lbrace \alpha, \beta \rbrace$. We also observe that for the function $\varphi$ satisfying the weak scaling property \eqref{a:WS}, we have $I_\varphi \subset [\sigma_1, \sigma_2]$. We may and do assume that $I_\varphi = [\sigma_1, \sigma_2]$ by considering the largest $\sigma_1$ and the smallest $\sigma_2$ such that \eqref{a:WS} holds.

The first main result is the Evans–Krylov type $C^{\varphi \psi}$ interior estimates for concave translation invariant nonlocal fully nonlinear equations with rough kernels of variable orders. The precise definition of generalized H\"older spaces such as $C^{\varphi \psi}$ or $C^\psi$ will be given in \Cref{section:generalized_Holder_space}, but the basic idea is that the modulus of continuity $r^{\sigma+\alpha}$ or $r^\alpha$ for the H\"older spaces $C^{\sigma+\alpha}$ or $C^\alpha$ are replaced by $\varphi(r) \psi(r)$ or $\psi(r)$, respectively.

\begin{theorem} \label{thm:Evans-Krylov}
Let $0 < \lambda \leq \Lambda$, $a \geq 1$, and $\sigma_0 \in (0,2)$. There is a universal constant $\bar{\alpha} \in (0,1)$, depending only on $n$, $\lambda$, $\Lambda$, $a$, and $\sigma_0$, such that the following statement holds: let $\I$ be a concave translation invariant elliptic operator with respect to $\mathcal{L}_0(\varphi)$, and assume
\begin{equation} \label{a:index}
I_\varphi \subset [\sigma_0, 2),
I_\psi \subset (0, \bar{\alpha}),
I_{\varphi\psi} \cap \mathbb{N} = \emptyset,
m_\varphi + \bar{\alpha} \notin \mathbb{N}, ~\text{and}~ 
\lfloor m_\varphi + \bar{\alpha} \rfloor = \lfloor m_{\varphi \psi} \rfloor.
\end{equation}
If $u \in C^{\varphi \psi}(B_1) \cap C^\psi(\Rn)$ and $f \in C^\psi(\overline{B_1})$ satisfy $\I u = f$ in $B_1$, then
\begin{equation*}
\Ve u \Ve_{C^{\varphi \psi}(\overline{B_{1/2}})} \leq C \left( \Ve u \Ve_{C^\psi (\Rn)} + \Ve f \Ve_{C^\psi(\overline{B_1})} \right),
\end{equation*}
where $C$ is a universal constant depending only on $n, \lambda, \Lambda, a, \sigma_0, \psi$, and $m_{\varphi \psi} - \lfloor m_{\varphi \psi} \rfloor$.
\end{theorem}

The non-integer assumptions in \eqref{a:index} are common and inevitable because of the well known technical difficulty arising from the H\"older spaces.

The universal constants in \Cref{thm:Evans-Krylov} and upcoming theorems stay uniform as $\sigma_1$ and $\sigma_2$ approach to 2 because constants depend on $\sigma_0$, but not $\sigma_1$ nor $\sigma_2$. This implies that our results recover the classical Evans–Krylov and Schauder theory for second order fully nonlinear equations as limits. Notice that in this case the dependence on $m_{\varphi\psi} - \lfloor m_{\varphi \psi} \rfloor$ is absorbed into the dependence on $\psi$. Our main theorems provide new results even in the case of second order fully nonlinear equations since the data is given in the generalized H\"older sense. Moreover, the results are also new in the case of the fractional Laplacian type equations ($\varphi(r) = r^\sigma$) because we do not restrict ourselves to Bellman type operators.

The next result is the Schauder type $C^{\varphi\psi}$ estimates for non-translation invariant fully nonlinear equations with rough kernels of variable orders. As we mentioned in the previous section, we need to impose $C^\psi$ dependence in $x$ variable directly to the operator $\I$. For this purpose, we consider, for a fixed point $z$, the function
\begin{equation*}
\beta_{\I}(x, x') = \sup \frac{\ve \I(u, x) - \I(u, x') \ve}{\Ve u \Ve'_{C^{\varphi\psi}(\overline{B_r(z)})} + \Ve u \Ve_{L^\infty(\Rn)}}, \quad x, x' \in \overline{B_r(z)},
\end{equation*}
where the supremum is taken over the space of all nontrivial functions with $\Ve u \Ve'_{C^{\varphi\psi}(\overline{B_r(z)})} + \Ve u \Ve_{L^\infty(\Rn)} < +\infty$. The function $\beta_{\I}$ is a nonlocal analogue of the function $\beta_F$ in \cite{CC95}, and it measures the oscillation of $\I$ in the $x$ variable. We define $\I_z$ by $\I_z u(x) := \I(\tau_{x-z} u, z)$, where $\tau_z u(x) = u(x+z)$, which is an operator obtained by freezing coefficients of the operator $\I$. Note that $\I_z$ is translation invariant since $\I_z \tau_w u(x) = \I (\tau_{x-z} \tau_w u, z) = \I (\tau_{x+w-z} u, z) = \I_z u(x+w)$. The closedness of the operators $\I$ and $\I_z$ in the $C^\psi$ fashion is given by
\begin{equation} \label{a:Holder}
\beta_{\I - \I_z} (x, x') \leq A_0 \psi(\ve x-x' \ve) \quad \forall x, x' \in B_r(z), \quad\text{for every ball}~ B_r(z) \subset B_1.
\end{equation}
Notice that \eqref{a:Holder} corresponds to $[a_{ij}(\cdot) - a_{ij}(z)]_{C^\psi(\overline{B_r(z)})} \leq A_0$, or equivalently, $[a_{ij}]_{C^\psi(\overline{B_r(z)})} \leq A_0$, in the case of second order linear operator in a non-divergence form.

Another assumption we need is the regularity estimates for the model equations. We say that $\I_z$ satisfies the {\it Evans–Krylov type estimates in $B_r=B_r(z)$} if, for given $\alpha \in (0, m_\psi)$ and given functions $f \in C^\psi(\overline{B_r})$ and $v \in C^{\varphi \psi}(B_r) \cap C^\psi(\Rn)$, $u \in C^{\varphi+\alpha}(\Rn)$ solves the equation $\I_z (u+v) = f$ in $B_r$, then $u \in C^{\varphi\psi}(\overline{B_{r/2}})$ and
\begin{equation} \label{eq:EK_estimate}
[u]_{C^{\varphi \psi}(\overline{B_{r/2}})} \leq C \left( \Ve u \Ve_{C^{\varphi+\alpha}(\Rn)} + [f]_{C^\psi(\overline{B_r})} + \sup_{L \in \mathcal{L}} [Lv]_{C^\psi(\overline{B_r})} \right)
\end{equation}
for some universal constant $C$. The class $\mathcal{L}$ in \eqref{eq:EK_estimate} is $\mathcal{L}_0(\varphi)$ or $\mathcal{L}_\psi(\varphi)$ according to the operator $\I$, i.e., if $\I$ is elliptic with respect to $\mathcal{L}_0(\varphi)$ or $\mathcal{L}_\psi(\varphi)$, then $\mathcal{L} = \mathcal{L}_0(\varphi)$ or $\mathcal{L} = \mathcal{L}_\psi(\varphi)$, respectively. Here, the space $C^{\varphi+\alpha}$ is the generalized H\"older space with the modulus of continuity $\varphi(r) r^\alpha$ (see \Cref{section:generalized_Holder_space}).

We point out that we say that $\I_z$ satisfies the Evans–Krylov type estimate if solutions enjoy \eqref{eq:EK_estimate}, not the estimate in \Cref{thm:Evans-Krylov}. This is because the estimate of the form \eqref{eq:EK_estimate} is useful for later uses in the following theorems, as well as it implies the estimate in \Cref{thm:Evans-Krylov} (see \Cref{prop:EK_intermediate}). The concave non-translation invariant elliptic operators are, of course, examples of operators satisfying the Evans–Krylov type estimates.

\begin{theorem} \label{thm:Schauder}
Let $\bar{\alpha}$ be the constant in \Cref{thm:Evans-Krylov}, and assume \eqref{a:index}.  Let $\I$ be a non-translation invariant operator which is elliptic with respect to $\mathcal{L}_0(\varphi)$ and satisfies \eqref{a:Holder}. Suppose that $\I_z$ satisfies the Evans–Krylov type estimates in $B_r(z)$ for every ball $B_r(z) \subset B_1$. If $u \in C^{\varphi\psi}(B_1) \cap C^\psi(\Rn)$ and $f \in C^\psi(\overline{B_1})$ satisfy $\I(u, x) = f(x)$ in $B_1$, then
\begin{equation*}
\Ve u \Ve_{C^{\varphi\psi}(\overline{B_{1/2}})} \leq C \left( \Ve u \Ve_{C^\psi(\Rn)} + \Ve f \Ve_{C^\psi(\overline{B_1})} \right),
\end{equation*}
where $C$ is a universal constant depending only on $n, \lambda, \Lambda, a, \sigma_0, A_0, \psi$, and $m_{\varphi \psi} - \lfloor m_{\varphi \psi} \rfloor$.
\end{theorem}

We will see, in \Cref{section:counterexample}, that the uniform estimates in \Cref{thm:Evans-Krylov} and \Cref{thm:Schauder} would be false if solutions $u$ are assumed to be merely bounded, as in \cite{Ser15a}. However, if the operator $\I$ is elliptic with respect to the subclass $\mathcal{L}_\psi(\varphi) \subset \mathcal{L}_0(\varphi)$ of linear operators whose kernels satisfy
\begin{equation} \label{a:L_psi}
[K]_{C^\psi(\Rn \setminus B_r)} \leq \Lambda \frac{c_\varphi}{r^n \varphi(r) \psi(r)} \quad\text{for all}~ r > 0,
\end{equation}
then we obtain $C^{\varphi\psi}$ uniform estimates for merely bounded solutions. Notice that the condition \eqref{a:L_psi} generalizes \eqref{eq:class_Lalpha}.

\begin{theorem} \label{thm:Schauder_bdd_data}
Suppose that $\I$ is elliptic with respect to $\mathcal{L}_\psi(\varphi)$ and satisfies the same assumptions as in \Cref{thm:Schauder}. If $u \in C^{\varphi\psi}(B_1) \cap L^\infty(\Rn)$ and $f \in C^\psi(\overline{B_1})$ satisfy $\I(u, x) = f(x)$ in $B_1$, then
\begin{equation*}
\Ve u \Ve_{C^{\varphi\psi}(\overline{B_{1/2}})} \leq C \left( \Ve u \Ve_{L^\infty(\Rn)} + \Ve f \Ve_{C^\psi(\overline{B_1})} \right),
\end{equation*}
where $C$ is a universal constant depending only on $n, \lambda, \Lambda, a, \sigma_0, A_0, \psi$, and $m_{\varphi \psi} - \lfloor m_{\varphi \psi} \rfloor$.
\end{theorem}

The idea of proofs is based on a Liouville type theorem and a compactness argument using blow-up sequences. This argument has been used successfully to establish regularity theory for nonlocal equations. See, for examples, \cite{Ser15b}, \cite{ROS16a}, and \cite{ROS16b}. The argument heavily relies on the scaling invariance of the equations because it allows us to consider blow-up sequences and its limit. However, the kernels of operators in $\mathcal{L}_0(\varphi)$ are not homogeneous and hence our equations are do not have the scaling invariance. Main difficulty arises at this point in the scaling argument. We will see that the rescaled equation and rescaled solution are related to a new scale function, that is, rescaled ones behave differently at each scale. Even though the rescaled equation may be different from the original one, the weak scaling property will make the rescaled equations belong to the same class of equations with different scale functions, but with the same constants $\sigma_1$, $\sigma_2$, and $a$, and the same ellipticity constants $\lambda$ and $\Lambda$. In other words, the weak scaling property makes the rescaling procedure preserve the key features of the equations and solutions.

The paper is organized as follows. In \Cref{section:preliminaries}, we observe how the weak scaling property serves to rescale the equations and solutions. Moreover, the definition and properties of generalized H\"older spaces are provided. In \Cref{section:Liouville}, the Liouville type theorem is stated and proved, which will be the key ingredient of the proof of the Evans–Krylov type theorem. \Cref{section:EK} is devoted to the proof of \Cref{thm:Evans-Krylov}, where the compactness argument with blow-up sequences plays an important role. By using the Evans–Krylov type estimates and by ``freezing coefficients" of the equations, we establish the Schauder type estimates in \Cref{section:Sch}. Both \Cref{thm:Schauder} and \Cref{thm:Schauder_bdd_data} will be proved in this section. We finish the paper with counterexamples to $C^{\varphi \psi}$ interior regularity for merely bounded solutions in \Cref{section:counterexample}.

\section{Preliminaries} \label{section:preliminaries}

\subsection{Weak scaling property} \label{section:WS}

In this section we study how the rescaling procedure works and collect some useful inequalities that will be used frequently in the sequel.

Let $\varphi$ satisfy the weak scaling property \eqref{a:WS} with constant $0 < \sigma_1 \leq \sigma_2 < 2$ and $a \geq 1$. We first observe in the following proposition that if some equation is related to $\varphi$, then a rescaled equation is related to a new scale function $\bar{\varphi}$ which is defined, for given $\rho > 0$, by
\begin{equation} \label{eq:bar_varphi}
\bar{\varphi}(r) = \frac{\varphi(\rho r)}{\varphi(\rho)}.
\end{equation}
This means that the rescaling argument may break since the rescaled equation is not the same with the original equation. However, since $\bar{\varphi}$ satisfies the weak scaling property with the same constants $\sigma_1$, $\sigma_2$, and $a$, the rescaled equation and the original equation are of the same type. Thus, the following proposition can be used in obtaining uniform estimates that depend on the constants $\sigma_1$, $\sigma_2$, and $a$, but not on $\varphi$ itself.

\begin{proposition} \label{prop:rescaled_equation}
If $\I$ is elliptic with respect to $\mathcal{L}_0(\varphi)$, then the operator $\bar{\I}$, defined by
\begin{equation*}
\bar{\I}(\bar{u}, \bar{x}) := \varphi(\rho) \frac{c_{\bar{\varphi}}}{c_\varphi} \I \left( \bar{u}((\cdot - z)/\rho), z+\rho \bar{x} \right),
\end{equation*}
is elliptic with respect to $\mathcal{L}_0(\bar{\varphi})$ with the same ellipticity constants. Moreover, if $u$ is a solution of $\I (u, x) = f(x)$ in $B_\rho(z)$, then the function $\bar{u}$, defined by $\bar{u}(\bar{x}) = u(z+\rho \bar{x})$, solves $\bar{\I} (\bar{u}, \bar{x}) = \bar{f}(\bar{x})$ in $B_1$, where
\begin{equation*}
\bar{f}(\bar{x}) = \varphi(\rho) \frac{c_{\bar{\varphi}}}{c_\varphi} f(z + \rho \bar{x}).
\end{equation*}
\end{proposition}

\begin{proof}
For the first assertion, it is enough to show that
\begin{align} \label{eq:rescaled_Pucci_operator}
\begin{split}
\M^+_{\mathcal{L}_0(\bar{\varphi})} \bar{u}(\bar{x})
&= \varphi(\rho) \frac{c_{\bar{\varphi}}}{c_\varphi} \M^+_{\mathcal{L}_0(\varphi)} (\bar{u}((\cdot - z) / \rho)) (z+\rho \bar{x}) \quad\text{and} \\
\M^-_{\mathcal{L}_0(\bar{\varphi})} \bar{u}(\bar{x})
&= \varphi(\rho) \frac{c_{\bar{\varphi}}}{c_\varphi} \M^-_{\mathcal{L}_0(\varphi)} (\bar{u}((\cdot - z) / \rho)) (z+ \rho \bar{x}).
\end{split}
\end{align}
This follows from the simple change of variables: for $x = z+\rho \bar{x}$ and $y = \rho \bar{y}$,
\begin{align*}
\M^+_{\mathcal{L}_0(\bar{\varphi})} \bar{u}(\bar{x})
&= \int_\Rn \left( \Lambda \delta^+ (\bar{u}, \bar{x}, \bar{y}) - \lambda \delta^- (\bar{u}, \bar{x}, \bar{y}) \right) \frac{c_{\bar{\varphi}}}{\ve \bar{y} \ve^n \bar{\varphi}(\ve \bar{y} \ve)} d\bar{y} \\
&= \int_\Rn \left( \Lambda \delta^+ (\bar{u}, \bar{x}, y/\rho) - \lambda \delta^- (\bar{u}, \bar{x}, y/\rho) \right) \frac{c_{\bar{\varphi}}}{\ve y \ve^n \varphi(\ve y \ve)/\varphi(\rho)} dy \\
&= \varphi(\rho) \frac{c_{\bar{\varphi}}}{c_\varphi} \int_\Rn \left( \Lambda \delta^+ (\bar{u}((\cdot-z)/\rho), x,y) - \lambda \delta^- (\bar{u}((\cdot-z)/\rho), x, y) \right) \frac{c_\varphi}{\ve y \ve^n \varphi(\ve y \ve)} dy \\
&= \varphi(\rho) \frac{c_{\bar{\varphi}}}{c_\varphi} \M^+_{\mathcal{L}_0(\varphi)} (\bar{u}((\cdot - z) / \rho)) (x),
\end{align*}
and the same argument holds for $\M^-$. Thus $\bar{\I}$ is elliptic with respect to $\mathcal{L}_0(\bar{\varphi})$. The second assertion is obvious.
\end{proof}

When we rescale the equation for a while to apply known estimates and then rescale back, \Cref{prop:rescaled_equation} is very useful. However, it is not sufficient for the blow-up sequence argument. For this purpose, we need to study the limit behavior of the scale function \eqref{eq:bar_varphi} as $\rho \to 0$.

\begin{lemma} \label{lem:limit_varphi}
Let $\varphi$ satisfy the weak scaling property and \eqref{a:varphi}. Let $\lbrace \rho_j \rbrace$ be a sequence such that $\rho_j \searrow 0$ as $j \to \infty$, and set $\bar{\varphi}_j(r) = \varphi(\rho_j r) / \varphi(\rho_j)$. Then $\bar{\varphi}_j$ converges locally uniformly to some function $\bar{\varphi}$ that satisfies the weak scaling property with the same constants.
\end{lemma}

\begin{proof}
Fix $\varepsilon > 0$ and let $\delta > 0$ be a small constant to be determined. For $\ve x-x_0 \ve < \delta$, we have
\begin{equation*}
\ve \bar{\varphi}_j(x) - \bar{\varphi}_j(x_0) \ve =  \frac{\ve \varphi(\rho_j x) - \varphi(\rho_j x_0) \ve}{\varphi(\rho_j)} = \frac{\varphi'(\rho_j x_0^\ast)}{\varphi(\rho_j)} \ve \rho_j x - \rho_j x_0 \ve
\end{equation*}
for some point $x_0^\ast$ lying in between $x_0$ and $x$. By the assumption \eqref{a:varphi} and the weak scaling property, we obtain
\begin{equation*}
\frac{\varphi'(\rho_j x_0^\ast)}{\varphi(\rho_j)} \ve \rho_j x - \rho_j x_0 \ve \leq \frac{\varphi(\rho_j x_0^\ast)}{\varphi(\rho_j)} \frac{\ve x-x_0 \ve}{x_0^\ast} \leq a \max \lbrace (x_0^\ast)^{\sigma_1-1}, (x_0^\ast)^{\sigma_2-1} \rbrace \delta.
\end{equation*}
By taking $\delta$ sufficiently small so that $a \max \lbrace (x_0^\ast)^{\sigma_1-1}, (x_0^\ast)^{\sigma_2-1} \rbrace \delta < \varepsilon$, we conclude that $\lbrace \bar{\varphi}_j \rbrace$ is equicontinuous. Moreover, the weak scaling property shows that $\lbrace \bar{\varphi}_j \rbrace$ is locally uniformly bounded. Therefore, by the Arzel\`a–Ascoli theorem and the diagonal sequence argument, we find a subsequence of $\lbrace \bar{\varphi}_j \rbrace$, still denoted by $\lbrace \bar{\varphi}_j \rbrace$, that converges locally uniformly to some function $\bar{\varphi}$. The function $\bar{\varphi}$ enjoys the weak scaling property with the same constants. Indeed, we see that for $0 < r \leq R$,
\begin{equation*}
\frac{\bar{\varphi}(R)}{\bar{\varphi}(r)} = \frac{\lim_{j \to \infty} \bar{\varphi}_j(R)}{\bar{\varphi}(r)} \leq a \frac{\lim_{j \to \infty} \bar{\varphi}_j(r)}{\bar{\varphi}(r)} \left( \frac{R}{r} \right)^{\sigma_2} = a \left( \frac{R}{r} \right)^{\sigma_2}.
\end{equation*}
The lower bound of $\bar{\varphi}(R) / \bar{\varphi}(r)$ is obtained in the same way.
\end{proof}

Let us close this section with some useful inequalities for later uses.

\begin{lemma} \label{lem:c_varphi}
Let $\varphi$ satisfy the weak scaling property \eqref{a:WS} and let $\bar{\varphi}$ be given by \eqref{eq:bar_varphi} for $\rho \in (0,1)$. Then the following inequalities hold:
\begin{enumerate} [(i)]
\item
$a^{-1} (2-\sigma_2) \leq c_\varphi \leq a(2-\sigma_1)$.

\item
$\varphi(\rho) \frac{c_{\bar{\varphi}}}{c_\varphi} \leq C$ for some $C = C(a)$.
\end{enumerate}
\end{lemma}

\begin{proof}
See \cite[Lemma 2.3]{KL20} for the proof of (i). For (ii), we see that
\begin{equation*}
\varphi(\rho) \frac{c_{\bar{\varphi}}}{c_\varphi} = \left( \int_0^1 \frac{r}{\varphi(r)} dr \right) \left( \int_0^1 \frac{r}{\varphi(\rho r)} dr \right)^{-1} = \rho^2 \left( \int_0^1 \frac{r}{\varphi(r)} dr \right) \left( \int_0^\rho \frac{r}{\varphi(r)} dr \right)^{-1}.
\end{equation*}
Using the weak scaling property, we have
\begin{equation*}
\int_\rho^1 \frac{r}{\varphi(r)} dr \leq a \frac{\rho^{\sigma_1}}{\varphi(\rho)} \int_\rho^1 r^{1-\sigma_1} dr = a \frac{\rho^{\sigma_1}}{\varphi(\rho)} \frac{1-\rho^{2-\sigma_1}}{2-\sigma_1}
\end{equation*}
and
\begin{equation*}
\int_0^\rho \frac{r}{\varphi(r)} dr \geq \frac{1}{a} \frac{\rho^{\sigma_1}}{\varphi(\rho)} \int_0^\rho r^{1-\sigma_1} dr = \frac{1}{a} \frac{\rho^{\sigma_1}}{\varphi(\rho)} \frac{\rho^{2-\sigma_1}}{2-\sigma_1}.
\end{equation*}
Thus, we obtain
\begin{equation} \label{eq:varphi_rho}
\varphi(\rho) \frac{c_{\bar{\varphi}}}{c_\varphi} \leq \rho^2 \left( 1 + a^2 \frac{1-\rho^{2-\sigma_1}}{\rho^{2-\sigma_1}} \right) \leq \rho^2 + a^2 \rho^{\sigma_1} \leq 1 + a^2,
\end{equation}
which shows (ii).
\end{proof}

If we assume in addition that $\varphi$ is monotone non-decreasing (which is implied by the assumption that $\phi$ is a Bernstein function), then the proof of \Cref{lem:c_varphi} (ii) is reduced to
\begin{equation*}
\varphi(\rho) \frac{c_{\bar{\varphi}}}{c_\varphi} = \left( \int_0^1 \frac{r}{\varphi(r)} dr \right) \left( \int_0^1 \frac{r}{\varphi(\rho r)} dr \right)^{-1} \leq 1.
\end{equation*}

\subsection{Generalized H\"older space} \label{section:generalized_Holder_space}

This section is devoted to the generalized H\"older spaces. We adopt the definition of generalized H\"older spaces from \cite{BK15}. For more exposition of this space, see \cite{BK15} and references therein. Let $\Omega$ be an open subset of $\Rn$. Throughout the paper, $\Omega$ will denote either balls or the whole space $\Rn$, but the all definitions in this section can be made for arbitrary open set $\Omega$. Let $\psi$ be a function satisfying \eqref{a:psi}. The moduli of continuity we are mainly concerned with are $\psi$, $\varphi$, $\varphi\psi$, and $\varphi(r) r^\alpha$. In order to define the spaces $C^\psi(\overline{\Omega})$, we recall the definition \eqref{eq:index_psi} of the indices $M_\psi$ and $m_\psi$.

\begin{definition}
Suppose that $m_\psi \in (k, k+1]$ for some non-negative integer $k$. The Banach space $C^\psi(\overline{\Omega})$ is defined as the subspace of $C^k(\overline{\Omega})$, equipped with the norm
\begin{equation*}
\Ve u \Ve_{C^\psi(\overline{\Omega})} := \Ve u \Ve_{C^k(\overline{\Omega})} + [u]_{C^\psi(\overline{\Omega})} := \Ve u \Ve_{C^k(\overline{\Omega})} + \sup_{x, y \in \Omega, x \neq y} \frac{\ve D^k u(x) - D^k u(y) \ve}{\psi(\ve x-y \ve) \ve x-y \ve^{-k}}.
\end{equation*}
\end{definition}

Let us write $\Ve \cdot \Ve_{\psi; \Omega} = \Ve \cdot \Ve_{C^\psi(\overline{\Omega})}$ and $[\, \cdot \,]_{\psi; \Omega} = [\, \cdot \,]_{C^\psi(\overline{\Omega})}$ for the sake of brevity. We will also use the notation $C^\psi = C^\alpha$ instead of $C^{r^\alpha}$ when $\psi$ is a polynomial of power $\alpha$, with $\alpha \notin \mathbb{N}$ (If $\alpha \in \mathbb{N}$, then the space $C^\psi$ gives the Lipschitz space $C^{0,1}$, not $C^1$. However we will not make use of Lipschitz spaces in the paper). In particular, the generalized H\"older space with the modulus of continuity $\varphi(r) r^\alpha$ will be frequently used in this work, and in this case it will be denoted by $C^{\varphi + \alpha}$.

It is sometimes useful to introduce non-dimensional norms on $C^\psi(\overline{\Omega})$. If $\Omega$ is bounded, for $d = \mathrm{diam}\, \Omega$, we set
\begin{equation} \label{eq:adim_norm}
\Ve u \Ve'_{\psi; \Omega} := \Ve u \Ve'_{k; \Omega} + [u]'_{\psi; \Omega} := \Ve u \Ve'_{k; \Omega} + \psi(d) [u]_{\psi; \Omega}.
\end{equation}
The spaces $C^\psi(\overline{\Omega})$, equipped with the norm \eqref{eq:adim_norm}, are also Banach spaces. We will also make use of the interior norms
\begin{equation*}
\Ve u \Ve^\ast_{\psi; \Omega} := \Ve u \Ve^\ast_{k; \Omega} + [u]^\ast_{\psi; \Omega} := \Ve u \Ve^\ast_{k; \Omega} + \sup_{x, y \in \Omega, x \neq y} \psi(d_{x, y}) \frac{\ve D^k u(x) - D^k u(y) \ve}{\psi(\ve x-y \ve) \ve x-y \ve^{-k}},
\end{equation*}
where $d_{x, y} := \min \lbrace d_x, d_y \rbrace$ and $d_x = \mathrm{dist}(x, \partial \Omega)$. The space of functions in $C^\psi(\Omega)$ whose interior norms are finite is a Banach space, equipped with the norm $\Ve \cdot \Ve^\ast_{\psi; \Omega}$.

The following interpolation lemma will be used frequently in the sequel. The proof is given in \cite[Proposition 2.6]{BK15} for the whole space, but the same proof holds for balls.

\begin{lemma} \label{lem:interpolation}
Assume $I_{\psi_1}, I_{\psi_2} \subset (0,1) \cup (1,2) \cup (2,3)$ and $M_{\psi_1} < m_{\psi_2}$, and let $\varepsilon \in (0,1)$. Then there is a constant $C = C(n, \psi_1, \psi_2, \varepsilon) > 0$ such that
\begin{equation*}
\Ve u \Ve'_{\psi_1; B} \leq C \Ve u \Ve_{0; B} + \varepsilon \Ve u \Ve_{\psi_2; B},
\end{equation*}
for every ball $B = B_R(x_0)$.
\end{lemma}

We remark that, in \Cref{lem:interpolation}, the dependence of the constant $C$ on $\psi_1$ and $\psi_2$ can be elaborated into the dependence on $M_{\psi_1}$, $m_{\psi_1}$, $M_{\psi_2}$, $m_{\psi_2}$, and the constants appearing in the definitions of almost increasing and almost decreasing. Therefore, whenever we use \Cref{lem:interpolation} with $\psi_1 = \varphi$, the we can say that the constant $C$ depends on $\sigma_1$, $\sigma_2$, and $a$, or on $\sigma_0$ and $a$.

As we observed in the previous section, the rescaled operator and rescaled functions are related to a new scale function \eqref{eq:bar_varphi}. The following lemmas show how the norms of rescaled functions are related to the norms of the original function.

\begin{lemma} \label{lem:rescaled_norm}
Let $u \in C^\psi(\overline{B_\rho(z)})$ and define $\bar{u}(\bar{x}) = u(z+\rho \bar{x}) / \psi(\rho)$. Assume that $I_\psi \cap \mathbb{N} = \emptyset$ and set $\bar{\psi}(r) = \psi(\rho r) / \psi(\rho)$ for given $\rho > 0$. Then $\bar{u} \in C^{\bar{\psi}}(\overline{B_1})$ and $[\bar{u}]_{\bar{\psi}; B_1} = [u]_{\psi; B_\rho(z)}$.
\end{lemma}

\begin{proof}
Let $d = \lfloor m_\psi \rfloor$ be the integer part of $m_\psi$. Since $I_\psi = I_{\bar{\psi}}$, $d$ is also the integer part of $\bar{\psi}$. Thus, we have
\begin{align*}
[\bar{u}]_{\bar{\psi}; B_1}
&= \sup_{\bar{x}, \bar{y} \in B_1} \frac{\ve D^d \bar{u}(\bar{x}) - D^d \bar{u}(\bar{y}) \ve}{\bar{\psi}(\ve \bar{x} - \bar{y} \ve) \ve \bar{x} - \bar{y} \ve^{-d}} \\
&= \frac{\rho^d}{\psi(\rho)} \sup_{\bar{x}, \bar{y} \in B_1} \frac{\ve D^d u(z + \rho \bar{x}) - D^d u(z+ \rho \bar{y}) \ve}{\psi(\ve \rho\bar{x} - \rho\bar{y} \ve) \ve \rho \bar{x} - \rho \bar{y} \ve^{-d}} \rho^{-d} \psi(\rho) \\
&= \sup_{x, y \in B_\rho(z)} \frac{\ve D^d u(x) - D^d u(y) \ve}{\psi(\ve x-y \ve) \ve x-y \ve^{-d}} = [u]_{\psi; B_\rho(z)} < +\infty,
\end{align*}
where $x = z + \rho \bar{x}$ and $y = z + \rho \bar{y}$. Since $\Ve \bar{u} \Ve_{C^0(\overline{B_1})} = \frac{1}{\psi(\rho)} \Ve u \Ve_{C^0(\overline{B_\rho(z)})} < +\infty$, by \Cref{lem:interpolation}, we obtain $\bar{u} \in C^{\bar{\psi}}(\overline{B_1})$ with $[\bar{u}]_{\bar{\psi}; B_1} = [u]_{\psi; B_\rho(z)}$.
\end{proof}

\begin{lemma} \label{lem:rescaled_adimnorm}
Let $u \in C^\psi(\overline{B_\rho(z)})$ and define $\bar{u}(\bar{x}) = u(z+\rho \bar{x})$. Assume that $I_\psi \cap \mathbb{N} = \emptyset$ and set $\bar{\psi}(r) = \psi(\rho r) / \psi(\rho)$ for given $\rho > 0$. Then $\bar{u} \in C^{\bar{\psi}}(\overline{B_1})$ and $\Ve \bar{u} \Ve'_{\bar{\psi}; B_1} = \Ve u \Ve'_{\psi; B_\rho(z)}$.
\end{lemma}

\begin{proof}
As in the proof of \Cref{lem:rescaled_norm}, let $d = \lfloor m_\psi \rfloor$. Then, we have
\begin{align*}
\Ve \bar{u} \Ve'_{\bar{\psi}; B_1}
&= \sum_{i=0}^d 2^i \Ve D^i \bar{u} \Ve_{C^0(\overline{B_1})} + \bar{\psi}(2) \sup_{\bar{x}, \bar{y} \in B_1} \frac{\ve D^d \bar{u} (\bar{x}) - D^d \bar{u} (\bar{y}) \ve}{\bar{\psi}(\ve \bar{x} - \bar{y} \ve) \ve \bar{x} - \bar{y} \ve^{-d}} \\
&= \sum_{i=0}^d (2\rho)^i \Ve D^i u \Ve_{C^0(\overline{B_\rho(z)})} + \frac{\psi(2\rho)}{\psi(\rho)} \sup_{x, y \in B_\rho(z)} \frac{\ve D^d u(x) - D^d u(y) \ve}{\psi(\ve x-y \ve) \ve x-y \ve^{-d}} \psi(\rho) = \Ve u \Ve'_{\psi; B_\rho(z)},
\end{align*}
where $x = z +\rho \bar{x}$ and $y = z + \rho \bar{y}$, which gives the desired result.
\end{proof}

\section{Liouville theorem} \label{section:Liouville}

In this section, we prove the Liouville type theorem that is the key ingredient of the proof of the Evans–Krylov theorem. Before we state and prove the Liouville type theorem, we observe that if \eqref{a:index} is assumed for some $\bar{\alpha} \in (0,1)$, then there is $\alpha \in (0, m_\psi)$ such that
\begin{equation} \label{eq:alpha}
\lfloor m_{\varphi + \bar{\alpha}} \rfloor = \lfloor m_{\varphi \psi} \rfloor < m_{\varphi+\alpha} < m_{\varphi \psi}.
\end{equation}
Indeed, $\alpha := m_\psi - (m_{\varphi \psi} - \lfloor m_{\varphi \psi} \rfloor)/2 \in (0, m_{\psi})$ satisfies \eqref{eq:alpha}. Whenever we take $\alpha \in (0, m_\psi)$ in the paper, we have in mind this constant.

\begin{theorem} \label{thm:Liouville}
Let $0 < \lambda \leq \Lambda$, $a \geq 1$, and $\sigma_0 \in (0,2)$. There is a universal constant $\bar{\alpha} \in (0,1)$, depending only on $n$, $\lambda$, $\Lambda$, $a$, and $\sigma_0$, such that the following statement holds: assume \eqref{a:index} and let $\alpha \in (0, m_\psi)$ be the constant satisfying \eqref{eq:alpha}. Suppose that $u \in C^{\varphi + \alpha}_{\mathrm{loc}}(\Rn)$ satisfies the following properties.
\begin{enumerate} [(i)]
\item
There is a constant $C_1 > 0$ such that
\begin{equation*}
\Ve u \Ve'_{\varphi +\alpha; B_R} \leq C_1 \varphi(R) \psi(R) \quad\text{for all}~ R \geq 1.
\end{equation*}
\item
For any $h \in \Rn$, we have
\begin{equation*}
\M^-_{\mathcal{L}_0(\varphi)} (u(\cdot + h) - u) \leq 0 \leq \M^+_{\mathcal{L}_0(\varphi)} (u(\cdot + h) - u) \quad\text{in}~ \Rn.
\end{equation*}
\item
For every non-negative $L^1(\Rn)$ function $\mu$ with compact support, we have
\begin{equation*}
\M^+_{\mathcal{L}_0(\varphi)} \left( \fint u(\cdot + h) \mu(h) dh - u \right) \geq 0 \quad\text{in}~ \Rn,
\end{equation*}
where the symbol $\fint$ means the averaged integral $(\int \mu(h) dh)^{-1} \int$.
\end{enumerate}
Then $u$ is a polynomial of degree $d := \lfloor m_{\varphi\psi} \rfloor$.
\end{theorem}

As pointed out in \cite{Ser15a}, the growth condition (i) is not enough to have $\M^+_{\mathcal{L}_0(\varphi)} u$ and $\M^-_{\mathcal{L}_0(\varphi)} u$ well defined in the classical sense, but it guarantees that $\M^+_{\mathcal{L}_0(\varphi)}$ and $\M^-_{\mathcal{L}_0(\varphi)}$ of $u(\cdot + h) - u$ (and of $\fint u(\cdot +h)\mu(h) dh - u$) are well defined in the classical sense.

The proof of \Cref{thm:Liouville} basically follows the lines of the proof of \cite[Theorem 2.1]{Ser15a}, but we need to be careful in the scaling argument since the rescaled functions have different scales as we already observed in \Cref{section:preliminaries}. The following lemma shows this scaling procedure.

\begin{lemma} \label{lem:rescaled}
Under the same setting as in \Cref{thm:Liouville}, the rescaled function $\bar{u}(\bar{x}) = \frac{1}{\varphi(\rho) \psi(\rho)} u(\rho \bar{x})$ satisfies the same assumptions (i), (ii), and (iii), with $\varphi$ and $\psi$ replaced by $\bar{\varphi}$ and $\bar{\psi}$, respectively, where $\bar{\varphi}(r) = \varphi(\rho r) / \varphi(\rho)$ and $\bar{\psi}(r) = \psi(\rho r) / \psi(\rho)$.
\end{lemma}

\begin{proof}
By \Cref{lem:rescaled_adimnorm}, we have
\begin{equation*}
\Ve \bar{u} \Ve'_{\bar{\varphi}+\alpha; B_R} = \frac{1}{\varphi(\rho)\psi(\rho)} \Ve u \Ve_{\varphi+\alpha; B_{\rho R}} \leq C_1 \frac{\varphi(\rho R) \psi(\rho R)}{\varphi(\rho) \psi(\rho)} = C_1 \bar{\varphi}(R) \bar{\psi}(R),
\end{equation*}
which gives (i). The assumptions (ii) and (iii) follow from \eqref{eq:rescaled_Pucci_operator}. 
\end{proof}

The key step for the proof of \Cref{thm:Liouville} is to prove that $L_\varphi u \in C^{\bar{\alpha}}(\overline{B_{1/2}})$, where $L_\varphi$ is a linear operator of the form \eqref{eq:linear_operator} with the kernel $K_\varphi(y) := \frac{c_\varphi}{\ve y \ve^n \varphi(\ve y \ve)}$. We will prove that $P \leq C \ve h \ve^{\bar{\alpha}}$ and $N \leq C \ve h \ve^{\bar{\alpha}}$ for all $h \in \overline{B_1}$, where
\begin{align*}
P(h) &:= \int_\Rn \left( \delta(u, h, y) - \delta(u, 0, y) \right))_+ \frac{c_\varphi}{\ve y \ve^n \varphi(\ve y \ve)} dy \quad\text{and} \\
N(h) &:= \int_\Rn \left( \delta(u, h, y) - \delta(u, 0, y) \right))_- \frac{c_\varphi}{\ve y \ve^n \varphi(\ve y \ve)} dy,
\end{align*}
and that the same proof also works when the point 0 in the above definition is replaced by any points in $\overline{B_{1/2}}$, for some constant $C$ independent of $x$. Then it follows that $\ve L_\varphi u (x+h) - L_\varphi u(x) \ve \leq C \ve h \ve^{\bar{\alpha}}$ for all $x \in \overline{B_{1/2}}$ and $h \in \overline{B_1}$. We point out that the constant $C$ here is not necessarily independent of $\sigma_1$ and $\sigma_2$.

\begin{lemma} \label{lem:PN}
Under the same setting as in \Cref{thm:Liouville}, there is a constant $C > 0$ such that $P(x) \leq C \ve x \ve^{\bar{\alpha}}$ and $N(x) \leq C \ve x \ve^{\bar{\alpha}}$ for all $x \in \overline{B_1}$.
\end{lemma}

\begin{proof}
We first claim that, for all $R \geq 1$,
\begin{equation} \label{eq:P,N}
0 \leq P \leq C \psi(R) \quad \text{and} \quad 0 \leq N \leq C \psi(R) \quad\text{in}~ B_R.
\end{equation}
Let us provide the proof of \eqref{eq:P,N} for the most delicate case $d = 2$. The other cases $d = 0$ and $d = 1$ are obtained in a very similar way.

Let $x_1 = x \in B_1$ and $x_2 = 0$. If $y \in B_1$, then we have
\begin{equation*}
\ve \delta(u, x_1, y) - \delta(u, x_2, y) \ve \leq \ve y \ve^2 \ve D^2 u(x_1^\ast) - D^2 u(x_2^\ast) \ve,
\end{equation*}
where $x_i^\ast$, $i = 1, 2$, is a point lying in between $x_i+y$ and $x_i - y$. By the assumption (i), we obtain
\begin{align*}
\ve D^2 u(x_1^\ast) - D^2 u(x_2^\ast) \ve 
&\leq [u]_{\varphi+\alpha;B_3} \varphi(\ve x_1^\ast - x_2^\ast \ve) \ve x_1^\ast - x_2^\ast \ve^{\alpha-2} \\
&\leq C_1 \psi(3) 3^{-\alpha} \varphi(\ve x_1^\ast - x_2^\ast \ve) \ve x_1^\ast - x_2^\ast \ve^{\alpha-2}.
\end{align*}
Recall that $\alpha$ was chosen so that \eqref{eq:alpha} holds, which yields that $\varphi(r) r^{\alpha-2}$ is almost increasing. Since $\ve x_1^\ast - x_2^\ast \ve \leq 3$, we have $\varphi(\ve x_1^\ast - x_2^\ast \ve) \ve x_1^\ast - x_2^\ast \ve^{\alpha-2} \leq c^{-1} \varphi(3) 3^{\alpha-2} \leq ac^{-1} 3^\alpha$, with the help of the weak scaling property \eqref{a:WS}. Thus, we have
\begin{equation} \label{eq:B_1}
\ve \delta(u, x, y) - \delta(u, 0, y) \ve \leq C \ve y \ve^2.
\end{equation}
On the other hand, if $y \in \Rn \setminus B_1$, then by the assumption (i) and the weak scaling property \eqref{a:WS}, we obtain
\begin{align} \label{eq:B_1^c}
\begin{split}
\ve \delta(u, x, y) - \delta(u, 0, y) \ve 
&\leq \ve u(x+y) - u(y) \ve + \ve u(x-y) - u(y) \ve + 2 \ve u(x) - u(0) \ve \\
&\leq 4[u]_{1; B_{2\ve y \ve}} \ve x \ve \leq 4C_1 \varphi(\ve 2y \ve) \psi(\ve 2y \ve) \ve 2y \ve^{-1} \\
&\leq C \varphi(\ve y \ve) \psi(\ve 2y \ve) \ve y \ve^{-1}.
\end{split}
\end{align}
It follows from \eqref{eq:B_1}, \eqref{eq:B_1^c}, and \Cref{lem:c_varphi} (i), that 
\begin{align*}
P(x)
&\leq C \int_{B_1} \ve y \ve^2 \frac{c_\varphi}{\ve y \ve^n \varphi(\ve y \ve)} dy + C \int_{\Rn \setminus B_1} \varphi(\ve y \ve) \psi(\ve 2y \ve) \ve y \ve^{-1} \frac{c_\varphi}{\ve y \ve^n \varphi(\ve y \ve)} dy \\
&\leq C + C \int_1^\infty \psi(2r) r^{-2} dr.
\end{align*}
Since $M_\psi < \bar{\alpha}$, by definition of $M_\psi$, there is a small constant $\varepsilon \geq 0$ such that $M_\psi + \varepsilon < \bar{\alpha}$ and $r \to \psi(r) r^{-M_\psi - \varepsilon}$ is almost decreasing. Thus, for $r \geq 1$, we have $\psi(r) \leq C r^{M_\psi + \varepsilon}$, and hence $\psi(r) r^{-2}$ is integrable in $[1, \infty)$. Therefore, we arrive at $P \leq C$ in $B_1$, which proves the claim \eqref{eq:P,N} for $R = 1$. In order to prove \eqref{eq:P,N} for all $R \geq 1$, we consider the rescaled function $\bar{u}(\bar{x}) = \frac{1}{\varphi(\rho) \psi(\rho)} u(\rho \bar{x})$ for $\rho = R$. Then \Cref{lem:rescaled} shows that $\bar{u}$ satisfies the assumptions (i), (ii), and (iii), with the same constant $C_1$, but with $\varphi$ and $\psi$ replaced by $\bar{\varphi}$ and $\bar{\psi}$, respectively. By the same argument above, we have
\begin{equation*}
\bar{P}(\bar{x}) := \int_\Rn \left( \delta(\bar{u}, \bar{x}, \bar{y}) - \delta(\bar{u}, 0, \bar{y}) \right)_+ \frac{c_{\bar{\varphi}}}{\ve \bar{y} \ve^n \bar{\varphi}(\ve 
\bar{y} \ve)} d\bar{y} \leq C \quad\text{for}~ \bar{x} \in B_1.
\end{equation*}
Scaling back and using \Cref{lem:c_varphi} (i), we arrive at
\begin{align*}
P(x) 
&= \int_\Rn (\delta(u, x, y) - \delta(u, 0, y)))_+ \frac{c_\varphi}{\ve y \ve^n \varphi(\ve y \ve)} dy \\
&= \frac{c_\varphi}{c_{\bar{\varphi}}} \psi(\rho) \int_\Rn \left( \delta(\bar{u}, \bar{x}, \bar{y}) - \delta(\bar{u}, 0, \bar{y}) \right)_+ \frac{c_{\bar{\varphi}}}{\ve \bar{y} \ve^n \bar{\varphi}(\ve 
\bar{y} \ve)} d\bar{y} \\
&\leq a^2 \frac{2-\sigma_1}{2-\sigma_2} \psi(\rho) \bar{P}(\bar{x}) \leq C \psi(\rho) \quad\text{for}~ x \in B_\rho,
\end{align*}
and this proves \eqref{eq:P,N} for $P$. The bound for $N$ in \eqref{eq:P,N} can be obtained in a similar way. Note that the constant $C$ may depend on $\sigma_1$ and $\sigma_2$, but this constant is not relevant to the uniform estimates in the main theorems.

We have proved that
\begin{equation} \label{eq:P_rings}
0 \leq P \leq C\psi(2^k) \leq C 2^{k \bar{\alpha}} \quad\text{in}~ B_{2^k}(0)
\end{equation}
for all $k \geq 0$, since $\psi(2^k) \leq C (2^k)^{M_\psi+\varepsilon} \leq C 2^{k\bar{\alpha}}$. We next prove that \eqref{eq:P_rings} holds for all $k \leq 0$, from which the lemma will follow. If we show that $P \leq (1-\theta)C$ in $B_{1/2}$, then \eqref{eq:P_rings} for all $k$ will follow by scaling and iteration argument. Dividing $u$ by $C$, let us assume that $P \leq 1$ in $B_1$ and show that $P \leq 1-\theta$ in $B_{1/2}$.

Let $x_0 \in B_{1/2}$ be a point where the supremum of $P$ in $B_{1/2}$ is attained, and define the set
\begin{equation*}
A = \lbrace y \in \Rn: \delta(u, x, y) - \delta(u, 0, y) > 0 \rbrace.
\end{equation*}
We define
\begin{equation*}
v(x) := \int_A \left( \delta(u, x, y) - \delta(u, 0, y) \right) \frac{c_\varphi}{\ve y \ve^n \varphi(\ve y \ve)} dy
\end{equation*}
and define the set $D := \lbrace x \in B_1 : v \geq (1-\bar{\theta}) \rbrace$, where $\bar{\theta} = \lambda / (4\Lambda)$. We claim that there is a small constant $\eta > 0$ such that
\begin{equation} \label{eq:D}
\ve D \ve \leq (1-\eta) \ve B_1 \ve.
\end{equation}
Assume to the contrary that $\ve D \ve > (1-\eta) \ve B_1 \ve$ for some small constant $\eta$ to be determined later. Let us consider the function $w$ given by
\begin{equation*}
w(x) := \int_{\Rn \setminus A} \left( \delta(u, x, y) - \delta(u, 0, y) \right) \frac{c_\varphi}{\ve y \ve^n \varphi(\ve y \ve)} dy.
\end{equation*}
We approximate $\chi_{\Rn \setminus A}(y) \frac{c_\varphi}{\ve y \ve^n \varphi(\ve y \ve)}$ by $L^1$ functions $\mu$ with compact support, use the assumption (iii), and then use the stability result \cite[Lemma 4.3]{CS11b} to obtain that
\begin{equation} \label{eq:subsolution}
\M^+_{\mathcal{L}_0(\varphi)} w \geq 0 \quad\text{in}~ \Rn.
\end{equation}
Moreover, using the relation $v + w = P - N$, we have $0 \leq P-v \leq 1 - (1-\bar{\theta}) = \bar{\theta}$ in $D$. Since the assumption (ii) shows that $P$ and $N$ are comparable, i.e., $\frac{\lambda}{\Lambda} P \leq N \leq \frac{\Lambda}{\lambda} P$, we obtain
\begin{equation} \label{eq:w_D}
w = (P-v) - N \leq \bar{\theta} - \frac{\lambda}{\Lambda} P \leq \bar{\theta} - \frac{\lambda}{\Lambda} (1-\bar{\theta}) \leq - \frac{\lambda}{2\Lambda} =: -c \quad\text{in}~ D.
\end{equation}
Let us consider the function $\bar{w} = (w(r\cdot) + c)_+$ with $r > 0$ small. Then by \eqref{eq:subsolution}, $\M^+_{\mathcal{L}_0(\varphi)} \bar{w} \geq 0$ in $\Rn$, and we can make $\Ve \bar{w} \Ve_{L^1(\Rn, \omega)}$ as small as we want by taking sufficiently small $r$ and $\eta$, where $\omega$ is the weight function
\begin{equation*}\omega(y) = \frac{c_\varphi}{1+\ve y \ve^n \varphi(\ve y \ve)}.
\end{equation*}
Indeed, \eqref{eq:w_D} gives $\bar{w} = 0$ in $D/r$, which covers most of $B_{1/r}$, and \eqref{eq:P,N} allows us to control the weighted integral outside the ball $B_{1/r}$. Applying \Cref{thm:local_boundedness} to $\bar{w}$, we obtain $w(0) + c = \bar{w}(0) \leq c/2$, which yields a contradiction since $w(0) = 0$ by definition. Therefore, \eqref{eq:D} holds for some small constant $\eta > 0$.

We use the assumption (iii) and \cite[Lemma 4.3]{CS11b} again to approximate $\chi_A(y) \frac{c_\varphi}{\ve y \ve^n \varphi(\ve y \ve)}$ by $L^1$ functions $\mu$ with compact support. Consequently, we have $\M^+_{\mathcal{L}_0(\varphi)} v \geq 0$ in $\Rn$. Following the computation in the proof of \cite[Lemma 4.9]{KL20} and using \eqref{eq:P,N}, for given $\delta_0 > 0$ we can find $\bar{\alpha} \in (0,1)$ small enough so that $\M^-_{\mathcal{L}_0(\varphi)} \bar{v} \leq \delta_0$ in $B_{3/4}$, where $\bar{v} = (1-v)_+$. By applying \cite[Theorem 4.7]{KL20} to $\bar{v}$ and using \eqref{eq:D}, we obtain that
\begin{equation*}
\eta \ve B_1 \ve \leq \ve \lbrace (1-v)_+ > \bar{\theta} \rbrace \ve \leq C (1-v) \quad\text{in}~ B_{1/2},
\end{equation*}
which concludes that $v \leq 1-\eta/C =: 1-\theta$ in $B_{1/2}$.
\end{proof}

To finish the proof of \Cref{thm:Liouville}, we make use of the potential theory \cite{BK15} for linear integro-differential operators. The result \cite{BK15} only provides the estimates which are not robust with respect to the order of differentiability, but it is enough to conclude \Cref{thm:Liouville}.

\begin{proof} [Proof of \Cref{thm:Liouville}]
By \Cref{lem:PN}, we have $L_\varphi u \in C^{\bar{\alpha}}(\overline{B_{1/2}})$ and $\Ve L_\varphi u \Ve_{\bar{\alpha}; B_{1/2}} \leq C$. Thus, the potential theory shows that $\Ve u \Ve_{\varphi+\bar{\alpha}; B_{1/4}} \leq C$, and in particular, $[u]_{\varphi+\bar{\alpha}; B_{1/4}} \leq C$. We observed in \Cref{lem:rescaled} that the rescaled function $\bar{u}(\bar{x}) = \frac{1}{\varphi(\rho) \psi(\rho)} u(\rho \bar{x})$ satisfies the assumptions (i), (ii), and (iii), with $\varphi$ and $\psi$ replaced by $\bar{\varphi}$ and $\bar{\psi}$, respectively. Thus, the same argument above is applied to $\bar{u}$, resulting in $[\bar{u}]_{\bar{\varphi}+\bar{\alpha}; B_{1/4}} \leq C$. Scaling back, we arrive at $[u]_{\varphi+\bar{\alpha}; B_{\rho/4}} \leq C \psi(\rho) \rho^{-\bar{\alpha}}$ for all $\rho \geq 1$. Since $I_\psi \subset (0, \bar{\alpha})$, by taking limit $\rho \to +\infty$, we arrive at $[u]_{\varphi+\bar{\alpha}; \Rn} = 0$ which conclude that $u$ is a polynomial of degree $d(= \lfloor m_{\varphi \psi} \rfloor = \lfloor m_{\varphi + \bar{\alpha}} \rfloor)$.
\end{proof}

\section{Evans–Krylov type estimates} \label{section:EK}

We prove \Cref{thm:Evans-Krylov} in this section utilizing the Liouville type theorem. The following proposition is the key ingredient in the proof of \Cref{thm:Evans-Krylov}, and it will also be used for the Schauder type estimates in the next section.

\begin{proposition} \label{prop:EK_intermediate}
Let $\bar{\alpha}$ be the constant in \Cref{thm:Liouville}, assume \eqref{a:index}, and let $\alpha \in (0, m_\psi)$ satisfy \eqref{eq:alpha}. Let $\I$ be a concave translation invariant operator which is elliptic with respect to $\mathcal{L}_0(\varphi)$, and suppose that $f \in C^\psi(\overline{B_1})$ and $v \in C^{\varphi\psi}(\overline{B_1}) \cap L^\infty(\Rn)$ satisfy
\begin{equation*}
[f]_{\psi; B_1} + \sup_{L \in \mathcal{L}_0(\varphi)} [Lv]_{\psi; B_1} \leq C_0.
\end{equation*}
If $u \in C^{\varphi + \alpha}(\Rn)$ solves
\begin{equation*}
\I(u+v) = f \quad\text{in}~ B_1,
\end{equation*}
then $u \in C^{\varphi \psi}(\overline{B_{1/2}})$ and
\begin{equation*}
[u]_{\varphi \psi; B_{1/2}} \leq C \left( \Ve u \Ve_{\varphi + \alpha; \Rn} + C_0 \right),
\end{equation*}
where $C$ is a universal constant depending only on $n$, $\lambda$, $\Lambda$, $a$, $\sigma_0$, $\psi$, and $m_{\varphi \psi} - \lfloor m_{\varphi \psi} \rfloor$.
\end{proposition}

\begin{proof}
Assume to the contrary that for each integer $k \geq 0$, there exist $u_k$, $v_k$, $f_k$, and $\I_k$, such that $\I_k (u_k + v_k) = f_k$ in $B_1$ and
\begin{equation*}
[u_k]_{\varphi \psi; B_{1/2}} > k \left( \Ve u_k \Ve_{\varphi + \alpha; \Rn} + C_{0,k} \right),
\end{equation*}
where $C_{0,k} = [f_k]_{\psi; B_1} + \sup_{L \in \mathcal{L}_0(\varphi)} [Lv_k]_{\psi; B_1}$. We may always assume that
\begin{equation} \label{eq:normalizing}
\Ve u_k \Ve_{\varphi + \alpha; \Rn} + C_{0,k} = 1 \quad\text{and}\quad [u_k]_{\varphi\psi; B_{1/2}} > k,
\end{equation}
by considering $\bar{u}_k := K^{-1} u_k$, $\bar{v}_k := K^{-1} v_k$, $\bar{f}_k := K^{-1} f_k$, and $\bar{\I}_k u := K^{-1} \I_k (Ku)$ with $K = \Ve u_k \Ve_{\varphi + \alpha; \Rn} + C_{0,k}$, instead of $u_k, v_k, f_k$, and $\I_k$. Then we have
\begin{equation} \label{eq:supsupsup}
\sup_k \sup_{z \in B_{1/2}} \sup_{\rho > 0} \frac{\rho^\alpha}{\psi(\rho)} [u_k]_{\varphi + \alpha; B_\rho(z)} = +\infty.
\end{equation}
Indeed, if the left hand side of \eqref{eq:supsupsup} has a finite value, say $C_1$, then for all $z \in B_{1/2}$ and for all $\rho \in (0,1)$, we obtain
\begin{equation} \label{eq:supsupsup_finite}
\Ve D^d u_k(z+\cdot) - D^d u_k(z) \Ve_{L^\infty(B_\rho)} \leq [u_k]_{\varphi + \alpha; B_\rho(z)} \varphi(\rho) \rho^{\alpha - d} \leq C_1 \varphi(\rho) \psi(\rho) \rho^{-d},
\end{equation}
where $d := \lfloor m_{\varphi + \bar{\alpha}} \rfloor = \lfloor m_{\varphi + \alpha} \rfloor = \lfloor m_{\varphi\psi} \rfloor$. Since \eqref{eq:supsupsup_finite} contradicts to \eqref{eq:normalizing} for $k$ sufficiently large, the claim \eqref{eq:supsupsup} holds.

We define the function
\begin{equation*}
\theta(\rho) := \sup_k \sup_{z \in B_{1/2}} \sup_{\rho' > \rho} \frac{(\rho')^\alpha}{\psi(\rho')} [u_k]_{\varphi + \alpha; B(z, \rho')},
\end{equation*}
which is monotone non-increasing by definition and satisfies $\lim_{\rho \to 0} \theta(\rho) = +\infty$ by \eqref{eq:supsupsup}. Moreover, we know from $\Ve u_k \Ve_{\varphi + \alpha; \Rn} \leq 1$ and the assumption \eqref{eq:alpha} that $\theta(\rho) < +\infty$ for $\rho > 0$. For every positive integer $j$, there are $\rho_j \geq 1/j$, $k_j$, and $z_j \in B_{1/2}$ such that
\begin{equation} \label{eq:sequence_j}
\frac{1}{2} \theta(\rho_j) \leq \frac{1}{2} \theta(1/j) \leq \frac{\rho_j^\alpha}{\psi(\rho_j)} [u_{k_j}]_{\varphi  + \alpha; B(z_j, \rho_j)} \leq \theta(1/j).
\end{equation}
Note that we have $\rho_j \to 0$ as $j \to +\infty$ since otherwise $\frac{\rho_j^\alpha}{\psi(\rho_j)} [u_{k_j}]_{\varphi + \alpha; B(z_j, \rho_j)}$ stays bounded while $\theta(1/j)$ blows up as $j \to +\infty$. We define $p_j$ by
\begin{equation} \label{eq:p_j}
p_j := \mathrm{arg \, min}_{p \in \mathcal{P}_d} \int_{B(z_j, \rho_j)} (u_{k_j} (x) - p_j (x-z))^2 dx,
\end{equation}
where $\mathcal{P}_d$ denotes the linear space of polynomials whose degrees are at most $d$. In other words, the polynomial $p_j$ best fits $u_{k_j}$ in $B(z_j, \rho_j)$ by least squares. Let us consider a blow up sequence
\begin{equation*}
\bar{u}_j(\bar{x}) := \frac{u_{k_j} (z_j + \rho_j \bar{x}) - p_j(\rho_j \bar{x})}{\varphi(\rho_j) \psi(\rho_j) \theta(\rho_j)}.
\end{equation*}
Then it follows from \eqref{eq:p_j} that for all $j \geq 1$,
\begin{equation} \label{eq:optimality}
\int_{B_1} \bar{u}_j(x) q(x) dx = 0 \quad \text{for all}~ q \in \mathcal{P}_d,
\end{equation}
which is the optimality condition for the least squares. By \Cref{lem:rescaled_norm} and \eqref{eq:sequence_j}, we obtain for $\bar{\varphi}_j(r) = \varphi(\rho_j r) / \varphi(\rho_j)$,
\begin{align} \label{eq:u_j}
\begin{split}
[\bar{u}_j]_{\bar{\varphi}_j + \alpha; B_1} 
&= \frac{\rho_j^\alpha}{\psi(\rho_j) \theta(\rho_j)} \left[ \frac{u_{k_j}(z_j + \rho_j \cdot) - p_j(\rho_j \cdot)}{\varphi(\rho_j) \rho_j^\alpha} \right]_{\bar{\varphi}_j + \alpha; B_1} \\
&= \frac{\rho_j^\alpha}{\psi(\rho_j) \theta(\rho_j)} [u_{k_j}]_{\varphi+\alpha; B(z_j, \rho_j)} \geq \frac{1}{2},
\end{split}
\end{align}
where we have used the fact that $d = \lfloor m_{\bar{\varphi}_j + \alpha} \rfloor < m_{\bar{\varphi}_j + \alpha}$ and that $[p_j(\rho_j \cdot)]_{\bar{\varphi}_j + \alpha; B_1} = 0$. The assumptions \eqref{a:index} and \eqref{eq:alpha} are used here.

Recall that \Cref{lem:limit_varphi} shows that there is a subsequence of $\bar{\varphi}_j$, that we call again $\bar{\varphi}_j$, converging locally uniformly to some function $\bar{\varphi}$. We will prove that there is a subsequence of $\bar{u}_j$ converging in $C^{m_{\varphi}+\alpha-\varepsilon}_{\mathrm{loc}}(\Rn)$ to a function $\bar{u} \in C^{\bar{\varphi} + \alpha}_{\mathrm{loc}} (\Rn)$ for some $\varepsilon > 0$ small, and that $\bar{u}$ satisfies all the assumptions in \Cref{thm:Liouville} with $\varphi$ and $\psi$ replaced by $\bar{\varphi}$ and $r^{M_\psi}$, respectively. For this, let us prove the following uniform estimates of $\bar{u}_j$:
\begin{equation} \label{eq:unif_est}
\Ve \bar{u}_j \Ve'_{\bar{\varphi}_j + \alpha; B_R} \leq C \bar{\varphi}_j(R) R^{M_\psi} \quad\text{for all}~ R \geq 1.
\end{equation}
We use \Cref{lem:rescaled_norm}, the fact that $[p_j(\rho_j \cdot)]_{\bar{\varphi}_j + \alpha; B_R} = 0$, and the monotonicity of $\theta$, to obtain
\begin{equation*}
[\bar{u}_j]_{\bar{\varphi}_j + \alpha; B_R} = \frac{\rho_j^\alpha}{\psi(\rho_j) \theta(\rho_j)} [u_{k_j}]_{\varphi+\alpha; B(z_j, \rho_j R)} \leq \frac{\rho_j^\alpha}{\psi(\rho_j) \theta(\rho_j)} \frac{\psi(\rho_j R) \theta(\rho_j R)}{(\rho_j R)^\alpha} \leq \frac{\psi(\rho_j R)}{\psi(\rho_j)} R^{-\alpha}.
\end{equation*}
By the definition of $M_\psi$, we see that $\psi(\rho_j R) / \psi(\rho_j) \leq C R^{M_\psi}$. Thus, we obtain
\begin{equation} \label{eq:unif_est_highorder}
[\bar{u}_j]'_{\bar{\varphi}_j + \alpha; B_R} = \bar{\varphi}_j(R) R^\alpha [\bar{u}_j]_{\bar{\varphi}_j+\alpha; B_R} \leq C \bar{\varphi}_j(R) R^{M_\psi}.
\end{equation}
Moreover, \eqref{eq:unif_est_highorder} with $R = 1$ implies that $\Ve \bar{u}_j - p \Ve_{L^\infty(B_1)} \leq C$ for some $p \in \mathcal{P}_d$. Then \eqref{eq:optimality} gives $\Ve \bar{u}_j \Ve_{L^\infty(B_1)} \leq C$. Therefore, by using the interpolation inequality, the uniform estimates \eqref{eq:unif_est} is obtained, as in \cite{Ser15a}.

Since $m_{\bar{\varphi}_j + \alpha} = m_{\varphi + \alpha} \notin \mathbb{N}$, we have an inclusion $C^{\bar{\varphi}_j + \alpha}(\overline{B_R}) \subset C^{m_\varphi+\alpha}(\overline{B_R})$ with
\begin{equation} \label{eq:convergence}
\Ve u \Ve_{m_\varphi+\alpha; B_R} \leq C(a, R) \Ve u \Ve_{\bar{\varphi}_j+\alpha; B_R}.
\end{equation}
Indeed, since 
\begin{equation*}
\bar{\varphi}_j (\ve x-y \ve) \ve x-y \ve^\alpha \leq a \bar{\varphi}_j(2R) \left( \frac{\ve x-y \ve}{2R} \right)^{\sigma_1} \ve x-y \ve^\alpha \leq a^2 (2R)^{\sigma_2 - \sigma_1} \ve x-y \ve^{m_\varphi+\alpha},
\end{equation*}
we have
\begin{align*}
[u]_{m_\varphi+\alpha; B_R}
&= \sup_{x, y \in B_R} \frac{\ve D^d u(x) - D^d u(y) \ve}{\ve x-y \ve^{m_\varphi+\alpha-2}} \\
&\leq a (2R)^2 \sup_{x, y \in B_R} \frac{\ve D^d u(x) - D^d u(y) \ve}{\bar{\varphi}_j(\ve x-y \ve) \ve x-y \ve^{\alpha-2}} = a(2R)^2 [u]_{\bar{\varphi}_j+\alpha; B_R}.
\end{align*}
Thus, \eqref{eq:unif_est} and \eqref{eq:convergence} shows that for each $R \geq 1$, we find $C > 0$ such that $\Ve \bar{u}_j \Ve'_{m_\varphi+\alpha; B_R} \leq C$. Therefore, the Arzel\`a-Ascoli theorem and the standard diagonal sequence argument yields that there is a subsequence of $\bar{u}_j$ converging locally in $C^{m_\varphi+\alpha-\varepsilon}(\Rn)$ to some function $\bar{u} \in C^{\bar{\varphi}+\alpha}_{\mathrm{loc}}(\Rn)$, where $\varepsilon$ is a small constant such that $d < m_\varphi + \alpha - \varepsilon$.

Let us next check that $\bar{u}$ satisfies all the assumptions (i), (ii), and (iii), in \Cref{thm:Liouville}, with $\varphi$ and $\psi$ replaced by $\bar{\varphi}$ and $r^{M_\psi}$, respectively. The assumption (i) is obtained by simply passing to the limit \eqref{eq:unif_est}. In order to check (iii), we set
\begin{equation*}
\bar{w}_j(\bar{x}) := u_{k_j}(z_j + \rho_j \bar{x}) + v_{k_j}(z_j + \rho_j \bar{x}).
\end{equation*}
We know from \Cref{prop:rescaled_equation} that the function $\bar{w}_j$ satisfies
\begin{equation} \label{eq:w_j}
\bar{\I}_j \bar{w}_j = \bar{f}_j \quad\text{in} ~ B(-z_j/\rho_j, 1/\rho_j),
\end{equation}
where
\begin{align*}
\bar{\I}_j w (\bar{x}) 
&:= \varphi(\rho_j) \frac{c_{\bar{\varphi}_j}}{c_\varphi} \I_{k_j} (w((\cdot-z_j)/\rho_j)) (z_j + \rho_j \bar{x}) \quad\text{and} \\
\bar{f}_j (\bar{x})
&:= \varphi(\rho_j) \frac{c_{\bar{\varphi}_j}}{c_\varphi} f_{k_j} (z_j + \rho_j \bar{x}),
\end{align*}
and that $\bar{\I}_j$ is elliptic with respect to $\mathcal{L}_0(\bar{\varphi}_j)$. Let $\mu$ be a non-negative $L^1(\Rn)$ function with compact support. Then we have
\begin{equation} \label{eq:w_ellipticity}
\M^+_{\mathcal{L}_0(\bar{\varphi}_j)} \left( \fint \bar{w}_j(\cdot + \bar{h}) d\mu(\bar{h}) - \bar{w}_j \right)(\bar{x}) \geq \bar{\I}_j \left( \fint \bar{w}_j(\cdot + \bar{h}) d\mu(\bar{h}) \right)(\bar{x}) - \bar{\I}_j \bar{w}_j(\bar{x}).
\end{equation}
Since $\I$ is concave and translation invariant, so is $\bar{\I}_j$, and hence
\begin{align} \label{eq:concave_trans_inv}
\begin{split}
\bar{\I}_j \left( \fint \bar{w}_j(\cdot + \bar{h}) d\mu(\bar{h}) \right)(\bar{x}) - \bar{\I}_j \bar{w}_j(\bar{x})
&\geq \fint \bar{\I}_j (\bar{w}_j(\cdot + \bar{h}))(\bar{x}) d\mu(\bar{h}) - \bar{\I}_j \bar{w}_j(\bar{x}) \\
&= \fint \bar{\I}_j \bar{w}_j (\bar{x} + \bar{h}) d\mu(\bar{h}) - \bar{\I}_j \bar{w}_j(\bar{x}).
\end{split}
\end{align}
We take $j$ sufficiently large so that $\mathrm{supp} \, \mu \subset B(-z_j/\rho_j, 1/\rho_j)$, then by \eqref{eq:w_j} and \eqref{eq:normalizing}, we obtain
\begin{align} \label{eq:f_Holder}
\begin{split}
\fint \bar{\I}_j \bar{w}_j (\bar{x}+\bar{h}) d\mu(\bar{h}) - \bar{\I}_j \bar{w}_j(\bar{x}) 
&= \fint \left( \bar{f}_j (\bar{x} + \bar{h}) - \bar{f}_j(\bar{x}) \right) d\mu(\bar{h}) \\
&\geq -\varphi(\rho_j) \frac{c_{\bar{\varphi}_j}}{c_\varphi} \fint \psi(\ve \rho_j \bar{h} \ve) d\mu(\bar{h}).
\end{split}
\end{align}
Combining \eqref{eq:w_ellipticity}, \eqref{eq:concave_trans_inv}, and \eqref{eq:f_Holder}, we get
\begin{equation*}
-\varphi(\rho_j) \frac{c_{\bar{\varphi}_j}}{c_\varphi} \fint \psi(\ve \rho_j \bar{h} \ve) d\mu(\bar{h}) \leq \M^+_{\mathcal{L}_0(\bar{\varphi}_j)} \left( \fint \bar{w}_j(\cdot + \bar{h}) d\mu(\bar{h}) - \bar{w}_j \right)(\bar{x}).
\end{equation*}
Set
\begin{equation*}
\bar{v}_j(\bar{x}) := \frac{v_{k_j}(z_j + \rho_j \bar{x})}{\varphi(\rho_j) \psi(\rho_j) \theta(\rho_j)},
\end{equation*}
then we have $\bar{w}_j = \varphi(\rho_j) \psi(\rho_j) \theta(\rho_j) (\bar{u}_j + \bar{v}_j) + p_j(\rho_j \cdot)$. Since $d \leq 2$, we obtain
\begin{equation*}
\delta(p_j(\rho_j \cdot), \bar{x} + \bar{h}, \bar{y}) - \delta(p_j(\rho_j \cdot), \bar{x}, \bar{y}) = 0
\end{equation*}
for all $\bar{x}, \bar{y}, \bar{h} \in \Rn$, and hence,
\begin{align*}
- \frac{1}{\theta(\rho_j)}\frac{c_{\bar{\varphi}_j}}{c_\varphi} \fint \bar{\psi}_j (\ve \bar{h} \ve) d\mu(\bar{h})
\leq&\, \M^+_{\mathcal{L}_0(\bar{\varphi}_j)} \left( \fint (\bar{u}_j + \bar{v}_j)(\cdot + \bar{h}) d\mu(\bar{h}) - (\bar{u}_j + \bar{v}_j) \right)(\bar{x}) \\
\leq&\, \M^+_{\mathcal{L}_0(\bar{\varphi}_j)} \left( \fint \bar{u}_j(\cdot + \bar{h}) d\mu(\bar{h}) - \bar{u}_j \right)(\bar{x}) \\
&+ \M^+_{\mathcal{L}_0(\bar{\varphi}_j)} \left( \fint \bar{v}_j(\cdot + \bar{h}) d\mu(\bar{h}) - \bar{v}_j \right)(\bar{x})
\end{align*}
for all $\bar{x} \in B(-z_j / \rho_j, 1/\rho_j)$. Using \eqref{eq:rescaled_Pucci_operator} and \eqref{eq:normalizing}, we have for $x = z_j + \rho_j \bar{x} \in B_1$,
\begin{align}
&\M^+_{\mathcal{L}_0(\bar{\varphi}_j)} \left( \fint \bar{v}_j(\cdot + \bar{h}) d\mu(\bar{h}) - \bar{v}_j \right)(\bar{x}) \nonumber \\
&= \frac{1}{\psi(\rho_j) \theta(\rho_j)} \frac{c_{\bar{\varphi}_j}}{c_\varphi} \M^+_{\mathcal{L}_0(\varphi)} \left( \fint \left( v_{k_j}(\cdot + h) - v_{k_j} \right) d\mu \left( \frac{h}{\rho_j} \right) \right) (x) \nonumber \\
&\leq \frac{1}{\theta(\rho_j)} \frac{c_{\bar{\varphi}_j}}{c_\varphi} \fint \bar{\psi}_j (\ve \bar{h} \ve) d\mu(\bar{h}). \label{eq:sup_Lv}
\end{align}
Therefore, we obtain
\begin{equation} \label{eq:Mu}
- \frac{1}{\theta(\rho_j)} \frac{c_{\bar{\varphi}_j}}{c_\varphi} \fint \bar{\psi}_j (\ve \bar{h} \ve) d\mu(\bar{h}) \leq \M^+_{\mathcal{L}_0(\bar{\varphi}_j)} \left( \fint \bar{u}_j(\cdot + \bar{h}) d\mu(\bar{h}) - \bar{u}_j \right)(\bar{x}).
\end{equation}
Since $\theta(\rho_j) \to +\infty$, $c_{\bar{\varphi}_j} \to c_\varphi$, and $\fint \bar{\psi}_j(\ve \bar{h} \ve) d\mu(\bar{h}) \to \fint \bar{\psi} (\ve \bar{h} \ve) d\mu(\bar{h})$ as $j \to +\infty$, the left hand side of \eqref{eq:Mu} converges to 0 as $j \to +\infty$. For the right hand side of \eqref{eq:Mu}, we use the dominated convergence theorem to pass to the limit, which is guaranteed by the growth control \eqref{eq:unif_est}. Therefore, we arrive at
\begin{equation*}
0 \leq \M^+_{\mathcal{L}_0(\bar{\varphi})} \left( \fint \bar{u}(\cdot + \bar{h}) d\mu(\bar{h}) - u \right)\quad\text{in}~ \Rn,
\end{equation*}
which gives the assumption (iii). The assumption (ii) is obtained in a similar way.

Since we have checked all the assumptions of \Cref{thm:Liouville}, we conclude that $\bar{u}$ is a polynomial of degree $\lfloor m_{\bar{\varphi}} + M_\psi \rfloor$, which is equal to $d$ by the assumption \eqref{a:index}. Passing \eqref{eq:optimality} to the limit, we see that $\bar{u}$ is orthogonal to every polynomial of degree $d$ in $B_1$. This shows that $\bar{u}$ must be zero, which contradicts to the limit of \eqref{eq:u_j}. Therefore, we obtain the desired result.
\end{proof}

We are now ready to prove \Cref{thm:Evans-Krylov} by using \Cref{prop:EK_intermediate}.

\begin{proof} [Proof of \Cref{thm:Evans-Krylov}]
Let $z \in B_1$ and $\rho \in (0,1)$ be such that $B_\rho(z) \subset B_1$. By \Cref{prop:rescaled_equation}, the rescaled function $\bar{u}(\bar{x}) := u(z+\rho \bar{x})$ solves the equation $\bar{\I} \bar{u} = \bar{f}$ in $B_1$, where
\begin{equation*}
\bar{\I} \bar{u}(\bar{x}) := \varphi(\rho) \frac{c_{\bar{\varphi}}}{c_\varphi} \I (\bar{u}((\cdot-z)/\rho)) (z+\rho \bar{x})
\end{equation*}
is elliptic with respect to $\mathcal{L}_0(\bar{\varphi})$ and $\bar{f} (\bar{x}) := \varphi(\rho) \frac{c_{\bar{\varphi}}}{c_\varphi} f(z+\rho \bar{x})$. Let $\eta$ be a cut-off function such that $\eta = 1$ in $B_{3/4}$ and $\eta = 0$ outside $B_1$. Since the function $\eta \bar{u} \in C^{\varphi +\alpha}(\Rn)$ solves
\begin{equation*}
\bar{\I} (\eta \bar{u} + \bar{v}) = \bar{f} \quad\text{in}~ B_{1/2},
\end{equation*}
where $\bar{v} = (1-\eta)\bar{u}$, we have
\begin{equation*}
[\eta \bar{u}]_{\bar{\varphi} \bar{\psi}; B_{1/4}} \leq C \left( \Ve \eta \bar{u} \Ve_{\bar{\varphi} +\alpha; \Rn} + [\bar{f}]_{\bar{\psi}; B_{1/2}} + \sup_{L \in \mathcal{L}_0(\bar{\varphi})} [L \bar{v}]_{\bar{\psi}; B_{1/2}} \right)
\end{equation*}
by \Cref{prop:EK_intermediate}. We first claim that
\begin{equation} \label{eq:Lv}
\sup_{L \in \mathcal{L}_0(\bar{\varphi})} [L \bar{v}]_{\bar{\psi}; B_{1/2}} \leq C [\bar{u}]_{\bar{\psi}; \Rn}.
\end{equation}
Since $\bar{v} = (1-\eta) \bar{u} = 0$ in $B_{3/4}$, we have for $\bar{x}, \bar{x}' \in B_{1/2}$,
\begin{align*}
\ve L \bar{v}(\bar{x}) - L \bar{v}(\bar{x}') \ve
&\leq \int_{\Rn \setminus B_{1/4}} \ve \delta(\bar{v}, \bar{x}, \bar{y}) - \delta(\bar{v}, \bar{x}', \bar{y}) \ve \frac{c_{\bar{\varphi}}}{\ve \bar{y} \ve^n \bar{\varphi}(\ve \bar{y} \ve)} d\bar{y} \\
&\leq 4[\bar{v}]_{\bar{\psi}; \Rn} \bar{\psi}(\ve \bar{x}-\bar{x}' \ve) \int_{\Rn \setminus B_{1/4}} \frac{c_{\bar{\varphi}}}{\ve \bar{y} \ve^n \bar{\varphi}(\ve \bar{y} \ve)} d\bar{y}.
\end{align*}
We use \Cref{lem:c_varphi} to obtain that 
\begin{equation*}
\int_{\Rn \setminus B_{1/4}} \frac{c_{\bar{\varphi}}}{\ve \bar{y} \ve^n \bar{\varphi}(\ve \bar{y} \ve)} d\bar{y} \leq C (2-\sigma_1) \int_{1/4}^\infty r^{-1-\sigma_1} dr \leq C(n, a, \sigma_0),
\end{equation*}
which gives $\ve L \bar{v}(\bar{x}) - L \bar{v}(\bar{x}') \ve \leq C [\bar{u}]_{\bar{\psi}; \Rn} \bar{\psi}(\ve \bar{x}-\bar{x}' \ve)$. Thus, the claim \eqref{eq:Lv} is proved, and hence
\begin{equation*}
[\bar{u}]_{\bar{\varphi} \bar{\psi}; B_{1/4}} \leq C \left( \Ve \bar{u} \Ve_{\bar{\varphi} + \alpha; B_1} + [\bar{u}]_{\bar{\psi}; \Rn} + [\bar{f}]_{\bar{\psi}; B_{1/2}} \right).
\end{equation*}
Scaling back, we obtain
\begin{equation*}
[u]'_{\varphi \psi; B_{\rho/4}(z)} \leq C \left( \Ve u \Ve'_{\varphi + \alpha; B_\rho(z)} + \psi(\rho) [u]_{\psi; \Rn} + \varphi(\rho) \frac{c_{\bar{\varphi}}}{c_\varphi} [f]'_{\psi; B_{\rho/2}(z)} \right).
\end{equation*}
We know that $\varphi(\rho) \frac{c_{\bar{\varphi}}}{c_\varphi} \leq C$ by \Cref{lem:c_varphi} (ii) and that $\psi(\rho) \leq C$. Thus, with the help of the interpolation inequality with a small constant $\varepsilon > 0$, we have
\begin{equation} \label{eq:adimnorm}
\Ve u \Ve'_{\varphi \psi; B_{\rho/4}(z)} \leq C \left( \varepsilon \Ve u \Ve'_{\varphi\psi; B_\rho(z)} + \Ve u \Ve_{\psi; \Rn} + \Ve f \Ve'_{\psi; B_{\rho/2}(z)} \right).
\end{equation}
Note that \eqref{eq:adimnorm} can be written, by using the interior norms, as
\begin{equation*}
\Ve u \Ve^\ast_{\varphi \psi; B_1} \leq C \left( \varepsilon \Ve u \Ve^\ast_{\varphi \psi; B_1} + \Ve u \Ve_{\psi; \Rn} + \Ve f \Ve^\ast_{\psi; B_1} \right).
\end{equation*}
By taking $\varepsilon$ sufficiently small so that $C \varepsilon \leq 1/2$, we arrive at
\begin{equation*}
\Ve u \Ve^\ast_{\varphi \psi; B_1} \leq C \left( \Ve u \Ve_{\psi; \Rn} + \Ve f \Ve^\ast_{\psi; B_1} \right),
\end{equation*}
which concludes the theorem.
\end{proof}

\section{Schauder type estimates} \label{section:Sch}

In this section we establish the Schauder type estimates for non-translation invariant fully nonlinear equations. Both \Cref{thm:Schauder} and \Cref{thm:Schauder_bdd_data} will follow from the following intermediate statement that shows the ``freezing coefficients" step.

\begin{proposition} \label{prop:Sch_intermediate}
Let $\bar{\alpha}$ be the constant in \Cref{thm:Evans-Krylov}, assume \eqref{a:index}, and let $\alpha \in (0, m_\psi)$ satisfy \eqref{eq:alpha}. Let $\mathcal{L}$ be either $\mathcal{L}_0(\varphi)$ or $\mathcal{L}_\psi(\varphi)$, and let $\I$ be a non-translation invariant operator which is elliptic with respect to $\mathcal{L}$. Suppose that 
\begin{equation} \label{eq:Holder}
\beta_{\I - \I_0}(x, x') \leq A_0 \psi(\ve x-x' \ve) \quad\text{for all}~ x, x' \in B_1(0),
\end{equation}
with $A_0 \leq 1$, and that $\I_0$ satisfies the Evans–Krylov type estimates for $B_r(z) = B_1(0)$. If $u \in C^{\varphi\psi} (\overline{B_1}) \cap C^{\varphi + \alpha}(\Rn)$ solves, for $f \in C^\psi(\overline{B_1})$ and $v \in C^{\varphi\psi}(\overline{B_1}) \cap L^\infty(\Rn)$,
\begin{equation*}
\I(u + v, x) = f(x) \quad\text{in}~ B_1,
\end{equation*}
then
\begin{equation*}
[u]_{\varphi\psi; B_{1/2}} \leq C \left( \Ve u \Ve_{\varphi+\alpha; \Rn} + A_0 \Ve u+v \Ve_{\varphi \psi; B_1} + \Ve u+v \Ve_{L^\infty(\Rn)} + [f]_{\psi; B_1} + \sup_{L \in \mathcal{L}} [Lv]_{\psi; B_1} \right),
\end{equation*}
where $C$ is a universal constant depending only on $n, \lambda, \Lambda, a, \sigma_0, \psi$, and $m_{\varphi \psi} - \lfloor m_{\varphi\psi} \rfloor$.
\end{proposition}

If $\I$ is concave and elliptic with respect to $\mathcal{L}_0(\varphi)$, then the assumption that $\I_0$ satisfies the Evans–Krylov type estimates is fulfilled by \Cref{prop:EK_intermediate}. Note that this is also true if $\I$ is concave and elliptic with respect to $\mathcal{L}_\psi(\varphi)$ since \Cref{prop:EK_intermediate} holds true with $\sup_{L \in \mathcal{L}_0(\varphi)} [Lv]_{\psi; B_1}$ replaced by $\sup_{L \in \mathcal{L}_\psi(\varphi)} [Lv]_{\psi; B_1}$. It is easily checked by observing that the inequality \eqref{eq:sup_Lv} is the only part where $\sup_{L \in \mathcal{L}_0(\varphi)} [Lv]_{\psi; B_1}$ is used throughout the proof. Thus, concave non-translation invariant operators (elliptic with respect to $\mathcal{L}_0(\varphi)$ or $\mathcal{L}_\psi(\varphi)$) are examples of operators for \Cref{prop:Sch_intermediate}, and hence for \Cref{thm:Schauder} or \Cref{thm:Schauder_bdd_data}, respectively.

\begin{proof}
We write the equation $\I(u+v, x) = f(x)$ by
\begin{equation*}
\I_0 (u + v)(x) = f(x) - (\I(u+v, x) - \I_0(u+v)(x)) \quad\text{in}~ B_1,
\end{equation*}
Since $\I_0$ satisfies the Evans–Krylov type estimates in $B_1$, we have
\begin{equation} \label{eq:freezing}
[u]_{\varphi \psi; B_{1/2}} \leq C \left( \Ve u \Ve_{\varphi+\alpha; \Rn} + [f]_{\psi; B_1} + [\I(u+v, \cdot) - \I_0(u+v)]_{\psi; B_1} + \sup_{L \in \mathcal{L}(\varphi)} [Lv]_{\psi; B_1} \right).
\end{equation}
Thus, it only remains to estimate $[ \I (u+v, \cdot) - \I_0 (u+v) ]_{\psi; B_1}$. But, the assumption \eqref{eq:Holder} shows that
\begin{align} \label{eq:I-I_0}
\begin{split}
[\I(u+v, \cdot) - \I_0 (u+v)]_{\psi; B_1} 
&\leq \sup_{x, x' \in B_1} \frac{\beta_{\I-\I_0}(x, x')}{\psi(\ve x-x' \ve)} \left( \Ve u+v \Ve'_{\varphi \psi; B_1} + \Ve u+v \Ve_{L^\infty(\Rn)} \right) \\
&\leq A_0 \left( \Ve u+v \Ve_{\varphi \psi; B_1} + \Ve u+v \Ve_{L^\infty(\Rn)} \right).
\end{split}
\end{align}
Therefore, the result follows from \eqref{eq:freezing} and \eqref{eq:I-I_0} since $A_0 \leq 1$.
\end{proof}

We point out that the proofs of \Cref{thm:Schauder} and \Cref{thm:Schauder_bdd_data} differ only in the control of quantity $\sup_{L \in \mathcal{L}} [Lv]_{\psi; B_1}$ in \Cref{prop:Sch_intermediate}. It can be controlled by using the $C^\psi$ H\"older regularity of solutions $u$ in \Cref{thm:Schauder}. However, if both solutions and kernels are not regular enough, then this quantity cannot be controlled (see \Cref{section:counterexample}). This is why we require some regularity of kernels when solutions are merely bounded. Let us first prove \Cref{thm:Schauder}.

\begin{proof} [Proof of \Cref{thm:Schauder}]
Similarly as in the proof of \Cref{thm:Evans-Krylov}, let $B_{\rho/2} (z) \subset B_\rho(z) \subset B_1$ and consider the rescaled equation
\begin{equation} \label{eq:rescaled_eq_1/2}
\bar{\I}(\bar{u}, \bar{x}) = \bar{f}(\bar{x}) \quad\text{in}~ B_{1/2},
\end{equation}
where $\bar{u}$, $\bar{\I}$, and $\bar{f}$ are given as \Cref{prop:rescaled_equation}, with $\bar{\varphi}(r) := \varphi(\rho r) / \varphi(\rho)$ and $\bar{\psi}(r) := \psi(\rho r) / \psi(\rho)$. Then the assumption \eqref{a:Holder} reads as
\begin{equation*}
\beta_{\bar{\I} - \bar{\I}_0} (\bar{x}, \bar{x}') \leq A_0 \psi(\rho) \bar{\psi}(\ve \bar{x} - \bar{x}' \ve) \quad\text{for all}~ \bar{x}, \bar{x}' \in B_{1/2}.
\end{equation*}
Indeed, for $\bar{x}, \bar{x}' \in B_{1/2}$ and $\bar{w} \in C^{\bar{\varphi} \bar{\psi}} (\overline{B_{1/2}}) \cap L^\infty(\Rn)$, let $x = z + \rho \bar{x}$, $x' = z + \rho \bar{x}'$, and $w(x) = \bar{w}((x-z)/\rho)$. Then by \Cref{lem:rescaled_adimnorm} and \Cref{lem:c_varphi} (ii), we have
\begin{align*}
\beta_{\bar{\I} - \bar{\I}_0}(\bar{x}, \bar{x}')
&= \sup_{\bar{w}} \frac{\ve \bar{\I}(\bar{w}, \bar{x}) - \bar{\I}_0 \bar{w}(\bar{x}) - (\bar{\I}(\bar{w}, \bar{x}') - \bar{\I}_0 \bar{w}(\bar{x}')) \ve}{\Ve \bar{w} \Ve'_{\bar{\varphi} \bar{\psi}; B_{1/2}} + \Ve \bar{w} \Ve_{L^\infty(\Rn)}} \\
&= \varphi(\rho) \frac{c_{\bar{\varphi}}}{c_\varphi} \sup_w \frac{\ve \I (w, x) - \I_z w(x) - (\I (w, x') - \I_z w(x')) \ve}{\Ve w \Ve'_{\varphi \psi; B_{\rho/2}(z)} + \Ve w \Ve_{L^\infty(\Rn)}} \\
&\leq \beta_{\I - \I_z}(x, x') \leq A_0 \psi(\ve x-x' \ve) = A_0 \psi(\rho) \bar{\psi}(\ve \bar{x}-\bar{x}' \ve).
\end{align*}
Furthermore, $\bar{\I}_0$ has Evans–Krylov type estimate in $B_{1/2}$ since $\I_z$ has it in $B_{\rho/2}(z)$. To apply \Cref{prop:Sch_intermediate}, we make $\bar{A}_0 := A_0 \psi(\rho) \leq \varepsilon_0 \leq 1$ by taking $\rho = \rho(A_0, \psi) > 0$ sufficiently small. The universal constant $\varepsilon_0$ will be chosen later. Let $\eta$ be a cut-off function supported in $B_1$ satisfying $\eta \equiv 1$ on $B_{3/4}$ and write the equation \eqref{eq:rescaled_eq_1/2} as $\bar{\I}(\eta \bar{u} + \bar{v}, \bar{x}) = \bar{f}(\bar{x})$ with $\bar{v} = (1-\eta)\bar{u}$. Since $u \in C^{\varphi\psi}(B_1) \cap C^\psi(\Rn)$, we have $\bar{u} \in C^{\bar{\varphi}\bar{\psi}}(\overline{B_1}) \cap C^{\bar{\psi}}(\Rn)$, and hence $\eta \bar{u} \in C^{\bar{\varphi}\bar{\psi}}(\overline{B_{1/2}}) \cap C^{\bar{\varphi}+\alpha}(\Rn)$ and $\bar{v} \in C^{\bar{\varphi}\bar{\psi}}(\overline{B_{1/2}}) \cap C^{\bar{\psi}}(\Rn)$. Thus, we obtain by the rescaled version of \Cref{prop:Sch_intermediate},
\begin{equation*}
[\eta \bar{u}]_{\bar{\varphi} \bar{\psi}; B_{1/4}} \leq C \left( \Ve \eta \bar{u} \Ve_{\bar{\varphi}+\alpha; \Rn} + \bar{A}_0 \Ve \bar{u} \Ve_{\bar{\varphi} \bar{\psi}; B_{1/2}} + \Ve \bar{u} \Ve_{L^\infty(\Rn)} + [\bar{f}]_{\bar{\psi}; B_{1/2}} + \sup_{L \in \mathcal{L}_0(\bar{\varphi})} [L\bar{v}]_{\bar{\psi}; B_{1/2}} \right).
\end{equation*}
Notice that it follows from the computation \eqref{eq:Lv} that
\begin{equation*}
[\bar{u}]_{\bar{\varphi} \bar{\psi}; B_{1/4}} \leq C \left( \Ve \bar{u} \Ve_{\bar{\varphi}+\alpha; B_1} + \varepsilon_0 \Ve \bar{u} \Ve_{\bar{\varphi} \bar{\psi}; B_{1/2}} + \Ve \bar{u} \Ve_{\bar{\psi}; \Rn} + [\bar{f}]_{\bar{\psi}; B_{1/2}} \right).
\end{equation*}
By scaling back and then using \Cref{lem:c_varphi} (ii) and $\psi(\rho) \leq C$, we have
\begin{equation*}
[u]'_{\varphi \psi; B_{\rho/4}(z)} \leq C \left( \Ve u \Ve'_{\varphi+\alpha; B_\rho(z)} + \varepsilon_0 \Ve u \Ve'_{\varphi\psi; B_{\rho/2}(z)} + \Ve u \Ve_{\psi; \Rn} + [f]'_{\psi; B_{\rho/2}(z)} \right).
\end{equation*}
Using the interpolation inequalities, we obtain that
\begin{equation*}
\Ve u \Ve'_{\varphi \psi; B_{\rho/4}(z)} \leq C \left( 2 \varepsilon_0 \Ve u \Ve'_{\varphi\psi; B_{\rho/2}(z)} + \Ve u \Ve_{\psi; \Rn} + \Ve f \Ve'_{\psi; B_{\rho/2}(z)} \right),
\end{equation*}
or, in terms of the interior norms, that
\begin{equation*}
\Ve u \Ve^\ast_{\varphi \psi; B_1} \leq C \left( 2 \varepsilon_0 \Ve u \Ve^\ast_{\varphi\psi; B_1} + \Ve u \Ve_{\psi; \Rn} + \Ve f \Ve^\ast_{\psi; B_1} \right).
\end{equation*}
By taking $\varepsilon_0$ sufficiently small so that $2C\varepsilon_0 \leq 1/2$, we arrive at the desired estimates.
\end{proof}

We finally prove the last theorem. Instead of using the $C^\psi$ H\"older regularity of solutions, we use the regularity of kernels \eqref{a:L_psi} to estimate the quantity $\sup_{L \in \mathcal{L}_\psi(\varphi)} [Lv]_{\psi; B_1}$.

\begin{proof} [Proof of \Cref{thm:Schauder_bdd_data}]
By the same argument as in \Cref{thm:Schauder}, we have the rescaled equation $\bar{\I}(\eta \bar{u} + \bar{v}, \bar{x}) = \bar{f}(\bar{x})$ in $B_{1/2}$. In this case, we have $\eta \bar{u} \in C^{\bar{\varphi} \bar{\psi}}(\overline{B_{1/2}}) \cap C^{\bar{\varphi}+\alpha}(\Rn)$ and $\bar{v} \in C^{\bar{\varphi} \bar{\psi}}(\overline{B_{1/2}}) \cap L^\infty(\Rn)$. Thus, we can still apply the rescaled version of \Cref{prop:Sch_intermediate} to obtain
\begin{equation*}
[\eta \bar{u}]_{\bar{\varphi} \bar{\psi}; B_{1/4}} \leq C \left( \Ve \eta \bar{u} \Ve_{\bar{\varphi}+\alpha; \Rn} + \bar{A}_0 \Ve \bar{u} \Ve_{\bar{\varphi} \bar{\psi}; B_{1/2}} + \Ve \bar{u} \Ve_{L^\infty(\Rn)} + [\bar{f}]_{\bar{\psi}; B_{1/2}} + \sup_{L \in \mathcal{L}_\psi(\varphi)} [L \bar{v}]_{\bar{\psi}; B_{1/2}} \right).
\end{equation*}
Let us estimate $\sup_{L \in \mathcal{L}_\psi(\varphi)} [L\bar{v}]_{\bar{\psi}; B_{1/2}}$. Since $\bar{v} \equiv 0$ in $B_{3/4}$, we see that for $\bar{x} \in B_{1/2}$ and $\bar{h} \in B_{1/16}$,
\begin{align*}
\ve L\bar{v}(\bar{x}+ \bar{h}) - L\bar{v}(\bar{x}) \ve
&= \lv 2\int_{\Rn} \bar{v}(\bar{x}+\bar{y}) (K(\bar{y} - \bar{h}) - K(\bar{y})) d\bar{y} \rv \\
&\leq 2\psi(\ve \bar{h} \ve) \int_{\Rn \setminus B_{1/8}} \ve \bar{v} (\bar{x} + \bar{y}) \ve \frac{\ve K(\bar{y} - \bar{h}) - K(\bar{y}) \ve}{\psi(\ve \bar{h} \ve)} d\bar{y} \\
&\leq 2\Ve \bar{v} \Ve_{L^\infty(\Rn)} \psi(\ve \bar{h} \ve) \int_{\Rn \setminus B_{1/8}} \frac{\ve K(\bar{y} - \bar{h}) - K(\bar{y}) \ve}{\psi(\ve \bar{h} \ve)} d\bar{y}.
\end{align*}
The assumption \eqref{a:L_psi} implies that 
\begin{equation*}
\int_{\Rn \setminus B_{\rho_0}} \frac{\ve K(y) - K(y-h) \ve}{\psi(\ve h \ve)} dy \leq C
\end{equation*}
whenever $\ve h \ve \leq \rho_0/2$. Therefore, we obtain $\ve L\bar{v}(\bar{x}+ \bar{h}) - L\bar{v}(\bar{x}) \ve \leq C \Ve \bar{u} \Ve_{L^\infty(\Rn)} \psi(\ve \bar{h} \ve)$, and hence,
\begin{equation*}
[\bar{u}]_{\bar{\varphi}\bar{\psi}; B_{1/4}} \leq C \left( \Ve \bar{u} \Ve_{\bar{\varphi}+\alpha; B_1} + \varepsilon_0 \Ve \bar{u} \Ve_{\bar{\varphi} \bar{\psi}; B_{1/2}} + \Ve \bar{u} \Ve_{L^\infty(\Rn)} + [\bar{f}]_{\bar{\psi}; B_{1/2}} \right).
\end{equation*}
As in the proof of \Cref{thm:Schauder}, the standard covering argument finishes the proof.
\end{proof}

\section{Counterexamples to \texorpdfstring{$C^{\varphi\psi}$}{} regularity for merely bounded solutions} \label{section:counterexample}

In this section we observe that the $C^{\psi}(\Rn)$ assumption on $u$ in \Cref{thm:Evans-Krylov} and \Cref{thm:Schauder} can not be relaxed to $L^\infty(\Rn)$. This can be seen by finding a sequence $\lbrace u_m \rbrace$ of solutions to equations with rough kernels of variable orders in $B_1$ that satisfy $\Ve u_m \Ve_{L^\infty(\Rn)} \leq C$ and $\Ve u_m \Ve_{C^{\varphi\psi}(B_{1/2})} \to + \infty$ as $m \to +\infty$, under the assumption \eqref{a:index}. This sequence is given in \cite[Section 5]{Ser15a} for the case $\varphi = r^\sigma$ and $\psi = r^\alpha$. The construction of the sequence in a general framework is almost the same, but let us show how to construct the sequence to make the paper self-contained.

Let us find a sequence in dimension $n=1$, but this will give the sequence in every dimension by considering rotationally symmetric functions. For every $m \geq 1$, we consider the solution $u_m$ to
\begin{equation*}
\begin{cases}
\M^+_{\mathcal{L}_0(\varphi)} u_m = 0 &\text{in}~ (-1, 1), \\
u_m = 0 &\text{in}~ [-2,-1] \cup [1,2], \\
u_m = \mathrm{sign} \sin (m\pi x) &\text{in}~ (-\infty, -2] \cup [2, \infty).
\end{cases}
\end{equation*}
For $p > 0$ small enough, the function $\psi(x) = \mathrm{dist}(x, [-1/4, 1/4])^p$ satisfies $\M^+_{\mathcal{L}_0(\varphi)} \psi \leq 0$ in $(-1/4-\varepsilon, - 1/4) \cup (1/4, 1/4+\varepsilon)$ for some $\varepsilon > 0$. We use the translation of $\psi$ as a barrier and use the comparison principle \cite[Theorem 3.9]{KL20} to obtain $\ve u_m \ve \leq C \mathrm{dist}(x, (-\infty, -1] \cup [1, \infty))^p$ in $(-1, 1)$. Combining this with the interior H\"older estimates \cite[Theorem 1.2]{KL20}, we have
\begin{equation} \label{eq:sequence_Holder}
\Ve u_m \Ve_{C^\alpha([-2, 2])} = \Ve u_m \Ve_{C^\alpha([-1, 1])} \leq C
\end{equation}
for some $\alpha > 0$ small enough and $C > 0$ independent of $m$.

We next assume that
\begin{equation} \label{eq:contradiction}
\Ve u_m \Ve_{C^{\varphi \psi}((-1/2, 1/2))} \leq C,
\end{equation}
and find a contradiction by comparing two values $\M^+_{\mathcal{L}_0(\varphi)} u_m(0)$ and $\M^+_{\mathcal{L}_0(\varphi)} u_m(1/(2m))$, which are 0 by definition. Note that the maximal operator $\M^+_{\mathcal{L}_0(\varphi)}$ can be written by
\begin{align*}
\M^+_{\mathcal{L}_0(\varphi)} u(x) 
&= \int_\mathbb{R} \frac{\Lambda \delta_+ (u, x, y) - \lambda \delta_- (u, x, y)}{\ve y \ve \varphi(\ve y \ve)} dy \\
&= \int_\mathbb{R} \delta(u, x, y) \frac{\Lambda (\mathrm{sign}(\delta(u, x, y))_+ + \lambda (\mathrm{sign}(\delta(u, x, y)))_-}{\ve y \ve \varphi(\ve y \ve)} dy,
\end{align*}
where we omit the constant $c_\varphi$ in this section. Since for all $L \in \mathcal{L}_0(\varphi)$ the solution to $Lw = 0$ in $(-1, 1)$ with the same boundary data as $u_m$ satisfies $w(0) = 0$ and $\M^+_{\mathcal{L}_0(\varphi)} w \geq Lw = 0$, by the comparison principle we have $u_m(0) \geq 0$. This gives $\delta(u_m, 0, y) = -2u_m(0) \leq 0$ for $\ve y \ve > 2$. Thus, we have
\begin{equation*}
\M^+_{\mathcal{L}_0(\varphi)} u_m(0) = \int_{\mathbb{R}} \delta(u_m, 0, y) \frac{b_m(y)}{\ve y \ve \varphi(\ve y \ve)} dy,
\end{equation*}
where $b_m(y) = \Lambda (\mathrm{sign}(\delta(u_m, 0, y)))_+ + \lambda (\mathrm{sign}(\delta(u_m, 0, y)))_-$, and $b_m(y) \equiv \lambda$ for $\ve y \ve > 2$.

On the other hand, we have
\begin{equation*}
\delta\left( u_m, \frac{1}{2m}, y \right) = 2\mathrm{sign} \cos(m \pi y) - 2u_m\left( \frac{1}{2m} \right) \quad\text{for} ~ \ve y \ve > 2+\frac{1}{2m}.
\end{equation*}
Hence, we obtain
\begin{equation*}
\M^+_{\mathcal{L}_0(\varphi)} u_m \left( \frac{1}{2m} \right) = \int_{\mathbb{R}} \delta\left(u_m, \frac{1}{2m}, y\right) \frac{\tilde{b}_m(y)}{\ve y \ve \varphi(\ve y \ve)} dy,
\end{equation*}
where
\begin{equation*}
\tilde{b}_m(y) = \Lambda \left( \mathrm{sign} \left( \delta \left( u_m, \frac{1}{2m}, y \right) \right) \right)_+ + \lambda \left( \mathrm{sign} \left( \delta \left( u_m, \frac{1}{2m}, y \right) \right) \right)_-,
\end{equation*}
and
\begin{equation*}
\tilde{b}_m(y) = \lambda + \frac{\Lambda-\lambda}{2} (1+\mathrm{sign} \cos(m\pi y)) \quad\text{for}~ \ve y \ve > 2+\frac{1}{2m}.
\end{equation*}

We know from $\ve u_m \ve \leq 1$ in $\mathbb{R}$, $\ve \lbrace u_m < 0 \rbrace \cap (-5, 5) \ve \geq 1$, and $\M^+_{\mathcal{L}_0(\varphi)} u_m = 0$ in $(-1, 1)$, that $u_m \leq 1-\tau$ in $[-1/2, 1/2]$ for some $\tau > 0$ independent of $m$. Thus, using $u_m(0) \in [0, 1-\tau]$, we have for all $\gamma \in (0,1)$ and for all $m$,
\begin{equation*}
\int_{\ve y \ve > 2+\gamma} 2(\mathrm{sign} \cos(m \pi y) - u_m(0)) \frac{\frac{\Lambda-\lambda}{2} (1+ \mathrm{sign} \cos (m\pi y))}{\ve y \ve \varphi(\ve y \ve)} dy \geq 2c_1 > 0,
\end{equation*}
where $c_1$ is independent of $m$. For $m$ large enough so that
\begin{equation*}
\lv \int_{\ve y \ve > 2+\gamma} 2\mathrm{sign} \cos(m \pi y) \frac{\lambda}{\ve y \ve \varphi(\ve y \ve)} dy \rv \leq c_1, 
\end{equation*}
we obtain
\begin{equation*}
\int_{\ve y \ve > 2+\gamma} 2(\mathrm{sign} \cos(m \pi y) - u_m(0)) \frac{\tilde{b}_m(y)}{\ve y \ve \varphi(\ve y \ve)} dy \geq c_1 - \int_{\ve y \ve > 2+\gamma}u_m(0) \frac{\lambda}{\ve y \ve \varphi(\ve y\ve} dy.
\end{equation*}
Therefore, by using \eqref{eq:sequence_Holder}, we see that
\begin{align} \label{eq:B_2^c}
\begin{split}
&\int_{\ve y \ve > 2+\gamma} \delta\left( u_m, \frac{1}{2m}, y \right) \frac{\tilde{b}_m(y)}{\ve y \ve \varphi(\ve y \ve)} dy - \int_{\ve y \ve > 2+\gamma} \delta\left( u_m, 0, y \right) \frac{b_m(y)}{\ve y \ve \varphi(\ve y \ve)} dy \\
&\geq c_1 - 2 \lv u_m \left( \frac{1}{2m} \right) - u_m(0) \rv \int_{\ve y \ve > 2+\gamma} \frac{\tilde{b}_m(y)}{\ve y \ve \varphi(\ve y \ve)} dy \geq c_1 - C \lv \frac{1}{2m} - 0 \rv^\alpha.
\end{split}
\end{align}

We next prove that
\begin{equation} \label{eq:B_2}
\lv \int_{\ve y \ve < 2-\gamma} \delta\left( u_m, \frac{1}{2m}, y \right) \frac{\tilde{b}_m(y)}{\ve y \ve \varphi(\ve y \ve)} dy - \int_{\ve y \ve < 2-\gamma} \delta\left( u_m, 0, y \right) \frac{b_m(y)}{\ve y \ve \varphi(\ve y \ve)} dy \rv \leq C \left( \frac{1}{m} \right)^{\alpha'}
\end{equation}
for some $\alpha' \in (0, \alpha)$. By \eqref{eq:contradiction} and \eqref{eq:sequence_Holder}, we have for $\ve y \ve < 2-\gamma$,
\begin{equation*}
\lv \delta(u_m, 0, y) - \delta\left(u_m, \frac{1}{2m}, y\right) \rv \leq C \left( \varphi(\ve y \ve) \psi(\ve y \ve) \land \lv \frac{1}{2m} \rv^\alpha \right).
\end{equation*}
By the definition of $m_\psi$, we have $\psi(\ve y \ve) \leq C \ve y \ve^{m_\psi}$. Thus, for $\theta \in (0,1)$ and $\alpha' = (1-\theta)\alpha$, we obtain
\begin{equation} \label{eq:difference}
\lv \delta(u_m, 0, y) - \delta\left(u_m, \frac{1}{2m}, y\right) \rv \leq C_1 \varphi^\theta(\ve y \ve) \ve y \ve^{\theta m_\psi} \lv \frac{1}{2m} \rv^{\alpha'}.
\end{equation}
If we take $\theta$ sufficiently close to 1 so that $\theta m_\psi > (1-\theta)M_\varphi$, then we have
\begin{equation} \label{eq:integral}
\int_{\ve y \ve < 2-\gamma} \frac{\varphi^\theta(\ve y \ve) \ve y \ve^{\theta m_\psi}}{\ve y \ve^n \varphi(\ve y \ve)} dy \leq \int_{\ve y \ve < 2-\gamma} \frac{a^{1-\theta}}{\varphi^{1-\theta} (2)} \lv \frac{2}{y} \rv^{(1-\theta)\sigma_2} \ve y \ve^{-n+\theta m_\psi} dy \leq C.
\end{equation}
If $\delta(u_m, 0, y)$ or $\delta(u_m, 1/(2m), y)$ is greater than $2C_1 \varphi^\theta(\ve y \ve) \ve y \ve^{\theta m_\psi} \lv \frac{1}{2m} \rv^{\alpha'}$, then $b = \tilde{b} = \Lambda$. Similarly, if $\delta(u_m, 0, y)$ or $\delta(u_m, 1/(2m), y)$ is less than $-2C_1 \varphi^\theta(\ve y \ve) \ve y \ve^{\theta m_\psi} \lv \frac{1}{2m} \rv^{\alpha'}$, then $b = \tilde{b} = \lambda$. In these cases, we use \eqref{eq:difference} and \eqref{eq:integral} to compute the left hand side of \eqref{eq:B_2}. If both $\ve \delta(u_m, 0, y) \ve$ and $\ve \delta(u_m, 1/(2m), y) \ve$ are less than $2C_1 \varphi^\theta(\ve y \ve) \ve y \ve^{\theta m_\psi} \lv \frac{1}{2m} \rv^{\alpha'}$, then \eqref{eq:integral} is enough to conclude \eqref{eq:B_2}.

Since $\M^+_{\mathcal{L}_0(\varphi)} u_m (0) = 0$ and $\M^+_{\mathcal{L}_0(\varphi)} u_m (1/(2m)) = 0$, by \eqref{eq:B_2^c} and \eqref{eq:B_2}, we arrive at
\begin{equation*}
0 = \M^+_{\mathcal{L}_0(\varphi)} u_m \left(\frac{1}{2m} \right) - \M^+_{\mathcal{L}_0(\varphi)} u_m (0) \geq c_1 - C \lv \frac{1}{2m} \rv^\alpha - C \lv \frac{1}{2m} \rv^{\alpha'} - C\gamma.
\end{equation*}
By taking $\gamma < c_1 / C$ and taking limit $m \to \infty$, we get a contradiction.

\begin{appendix}

\section{Local boundedness} \label{section:appendix}

We provide a weak version of the mean value theorem, which gives a half Harnack inequality. The full Harnack inequality is proved by the authors \cite{KL20} and its proof combines the following theorem with the weak Harnack inequality which is the other half of the Harnack inequality.

\begin{theorem} \label{thm:local_boundedness}
Assume $\sigma_1 \geq \sigma_0 > 0$. Let $u$ be a function such that $u$ is continuous in $\overline{B_1}$ and $\Ve u \Ve_{L^1(\Rn, \omega)} \leq C_0$. If $\M^+_{\mathcal{L}_0(\varphi)} u \geq -C_0$ in $B_1$, then $u \leq CC_0$ in $B_{1/2}$ for some universal constant $C$, depending only on $n$, $\lambda$, $\Lambda$, $a$, and $\sigma_0$.
\end{theorem}

The proof of \Cref{thm:local_boundedness} is contained in the proof of \cite[Theorem 1.1]{KL20} but let us include it here to make the paper self-contained.

\begin{proof}
We may assume that $C_0 = 1$ by considering $u/C_0$ instead of $u$. Let $\varepsilon > 0$ be the constant in \cite[Theorem 4.7]{KL20} and let $\gamma = (n+2)/\varepsilon$. Let us consider the minimal value of $t > 0$ such that 
\begin{align*}
u(x) \leq h_t(x) := t (1-\ve x \ve)^{-\gamma} \quad \text{for all} ~ x \in B_{3/4}.
\end{align*}
Then there exists $x_0 \in B_{3/4}$ satisfying $u(x_0) = h_t(x_0)$. It is enough to show that $t$ is uniformly bounded.

Let $d = 1 - \ve x_0 \ve, r = d/2$, and let $A = \lb u > u(x_0) / 2 \rb$. The assumption $\Ve u \Ve_{L^1(\Rn, \omega)} \leq 1$ implies that $\Ve u \Ve_{L^1(B_1)} \leq C(n, \sigma_0)$, and hence
\begin{equation*}
\ve A \cap B_1 \ve \leq C \lv \frac{2}{u(x_0)} \rv \leq C t^{-1} d^\gamma \leq C t^{-1} d^n.
\end{equation*}
Since $B_r(x_0) \subset B_1$ and $r = d/2$, we obtain
\begin{align*}
\lv \lbrace u > u(x_0) / 2 \rbrace \cap B_r(x_0) \rv \leq C t^{-1} \ve B_r(x_0) \ve.
\end{align*}
Let us show that there exists a universal constant $\theta > 0$ such that
\begin{equation*}
\ve \lbrace u < u(x_0) / 2 \rbrace \cap B_{\theta r/4}(x_0) \ve \leq \frac{1}{2} \ve B_{\theta r/4} \ve
\end{equation*}
for $t > 1$ sufficiently large, which will yield that $t > 0$ is uniformly bounded.

Let us first estimate $\lv \lbrace u < u(x_0) / 2 \rbrace \cap B_{\theta r}(x_0) \rv$ for $\theta > 0$ sufficiently small, which will be chosen uniformly later. For every $x \in B_{\theta r}(x_0)$, we have
\begin{align*}
u(x) \leq h_t(x) \leq t \ve d-\theta r \ve^{-\gamma} = \left( 1- \theta/2 \right)^{-\gamma} u(x_0). 
\end{align*}
Let us consider the function
\begin{align*}
v(x) := \left( 1- \theta/2 \right)^{-\gamma} u(x_0) - u(x),
\end{align*}
which is non-negative only on $B_{\theta r}(x_0)$. In order to utilize \cite[Theorem 4.7]{KL20}, we consider the function $w := v_+$ instead of $v$, and thus, we need to compute $\mathcal{M}_{\mathcal{L}_0(\varphi)}^- w$ in $B_{\theta r/2}(x_0)$. For $x \in B_{\theta r/2}(x_0)$, since $\M^-_{\mathcal{L}_0(\varphi)} v(x) = - \M^+_{\mathcal{L}_0(\varphi)} u(x) \leq 1$, we have
\begin{align*}
\M^-_{\mathcal{L}_0(\varphi)} w(x)
&\leq \M_{\mathcal{L}_0(\varphi)}^- v(x) + \M_{\mathcal{L}_0(\varphi)}^+ v_-(x) \nonumber \\
&\leq 1 + 2\Lambda c_\varphi \int_{\Rn \cap \lbrace v(x+y) < 0 \rbrace} \frac{v_-(x+y)}{\ve y \ve^n \varphi(\ve y \ve)} \, dy \nonumber \\
&\leq 1 + 2\Lambda c_\varphi \int_{\Rn \setminus B_{\theta r/2}(x)} \frac{(u(x+y) - (1-\theta/2)^{-n} u(x_0))_+}{\ve y \ve^n \varphi(\ve y \ve)} dy \\
&\leq 1 + 2\Lambda c_\varphi \int_{\Rn \setminus B_{\theta r/2}(x)} \frac{\ve u(z) \ve}{\ve z-x \ve^n \varphi(\ve z-x \ve)} dz.
\end{align*}
There is a constant $c > 0$ such that $\ve z-x \ve \geq c \theta r (1+\ve z \ve)$ for all $z \in \Rn \setminus B_{\theta r/2}(x)$. Thus, by means of the weak scaling property \eqref{a:WS}, we have
\begin{equation*}
\M^-_{\mathcal{L}_0(\varphi)} w(x) \leq 1 + C \frac{c_\varphi}{(\theta r)^{n+2}} \int_{\Rn} \frac{\ve u(z) \ve}{1+\ve z \ve^n \varphi(\ve z \ve)} dz \leq 1 + \frac{C}{(\theta r)^{n+2}} \Ve u \Ve_{L^1(\Rn, \omega)} \leq \frac{C}{(\theta r)^{n+2}}.
\end{equation*}
We are now ready to apply \cite[Theorem 4.7]{KL20} to $w$ in $B_{\theta r/2}(x_0)$. Notice that, in our setting, \cite[Theorem 4.7]{KL20} reads as follows;
\begin{theorem} \label{thm:WHI}
Assume $\sigma_1 \geq \sigma_0 > 0$. If $u$ is a non-negative function in $\Rn$ such that $\M^-_{\mathcal{L}_0(\varphi)} u \leq C_0$ in $B_{2R}(x_0)$, then
\begin{equation*}
\ve \lbrace u > t \rbrace \cap B_R(x_0) \ve \leq C R^n \left( u(x_0) + C_0 \Phi(R) \right)^\varepsilon t^{-\varepsilon} \quad\text{for all} ~ t > 0,
\end{equation*}
for some universal constant $C$, depending only on $n$, $\lambda, \Lambda$, $a$, and $\sigma_0$, where
\begin{equation*}
\Phi(R) = \varphi(R) \left( \int_0^1 \frac{r}{\varphi(r)} dr \right) \left( \int_0^1 \frac{r\varphi(R)}{\varphi(rR)} dr \right)^{-1}.
\end{equation*}
\end{theorem}

Therefore, by \Cref{thm:WHI} and \Cref{lem:c_varphi} (i), we obtain
\begin{align*}
&\ve \lbrace u < u(x_0)/2 \rbrace \cap B_{\theta r/4}(x_0) \ve \\
&= \ve \lbrace w > ((1-\theta/2)^{-n} - 1/2) u(x_0) \rbrace \cap B_{\theta r/4}(x_0) \ve \\
&\leq C (\theta r/4)^n \left( ((1-\theta/2)^{-n} - 1) u(x_0) + C \frac{\Phi(\theta r/4)}{(\theta r)^{n+2}} \right)^\varepsilon \left( ((1-\theta/2)^{-n} - 1/2) u(x_0) \right)^{-\varepsilon}.
\end{align*}
We make the quantity $(1-\theta/2)^{-\gamma} -1/2$ bounded away from 0 by taking $\theta > 0$ sufficiently small. Moreover, we observe from \eqref{eq:varphi_rho} that
\begin{equation*}
\Phi(\theta r/4) \leq (\theta r/4)^2 + a^2 (\theta r/4)^{\sigma_1} \leq C (\theta r)^{\sigma_0}.
\end{equation*}
Therefore, recalling that $u(x_0) = t(2r)^{-\gamma}$ and that $\gamma > n+2$, we have
\begin{align*}
\ve \lbrace u < u(x_0)/2 \rbrace \cap B_{\theta r/4}(x_0) \ve
&\leq C (\theta r/4)^n \left( ((1-\theta/2)^{-n} - 1) + \theta^{-n+\sigma_0 - 2} r^{\gamma-(n+2)+\sigma_0} t^{-1} \right)^\varepsilon \\
&\leq C (\theta r/4)^n \left( ((1-\theta/2)^{-n} - 1)^\varepsilon + \theta^{(-n+\sigma_0 - 2)\varepsilon} t^{-\varepsilon} \right).
\end{align*}
Let us take $\theta > 0$ sufficiently small so that
\begin{align*}
C(\theta r)^n \left( \left( 1 - \theta/2 \right)^{-n} - 1 \right)^\varepsilon \leq \frac{1}{4} \lv B_{\theta r/4}(x_0) \rv.
\end{align*}
If $t > 0$ is sufficiently large, then we have
\begin{align*}
C(\theta r)^n \theta^{-n \varepsilon} t^{-\varepsilon} \leq \frac{1}{4} \lv B_{\theta r/4}(x_0) \rv,
\end{align*}
which implies that 
\begin{align*}
\lv \lbrace u < u(x_0)/2 \rbrace \cap B_{\theta r/4}(x_0) \rv \leq \frac{1}{2} \lv B_{\theta r/4}(x_0) \rv.
\end{align*}
Therefore, $t$ is uniformly bounded and the result follows.
\end{proof}

\end{appendix}

\section*{Acknowledgement}

The research of Minhyun Kim is supported by the National Research Foundation of Korea (NRF) grant funded by the Korea government (MSIP) : NRF-2016K2A9A2A13003815. The research of Ki-Ahm Lee is supported by the National Research Foundation of Korea(NRF) grant funded by the Korea government(MSIP) : NRF-2020R1A2C1A01006256.


\end{document}